\documentclass[a4paper,11pt,oneside]{article}
\usepackage[utf8]{inputenc}
\usepackage[T1]{fontenc}
\usepackage[english]{babel}
\usepackage{amsthm}
\usepackage{amssymb}
\usepackage{amsmath}
\usepackage{mathrsfs}
\usepackage{graphicx}
\usepackage{xcolor}
\usepackage{eurosym}
\usepackage{dsfont}
\usepackage{pifont}
\usepackage{fancybox}
\usepackage{multicol}
\usepackage{enumitem}
\usepackage{url}
\usepackage{bbm}
\usepackage[all]{xy}
\usepackage{authblk}
\usepackage{mathtools}
\usepackage{mathabx}
\reversemarginpar
\numberwithin{equation}{section}
\usepackage[top=3.2cm, bottom=3.3cm, left=2.5cm , right=2.5cm]{geometry}
\usepackage[colorlinks=true,linkcolor=black,citecolor=black,urlcolor=black,hyperfootnotes=false]{hyperref}

\usepackage[titletoc,toc,page,title,header]{appendix}

\theoremstyle{definition} \newtheorem{Def}{Definition}
\theoremstyle{definition}\newtheorem{Rem}{Remark}
\theoremstyle{definition}
\theoremstyle{plain}\newtheorem{theorem}{Theorem}
\theoremstyle{plain}
\theoremstyle{plain}\newtheorem{Lem}{Lemma}
\theoremstyle{plain}
\theoremstyle{plain} \newtheorem{Assu}{Assumption}

\newcommand{\R}{\mathbb{R}}

\newcommand{\Z}{\mathbb{Z}}

\newcommand{\Prob}{\mathbb{P}}
\newcommand{\E}[1]{\mathbb{E}\left[#1\right]}

\newcommand{\abs}[1]{\left\lvert#1\right\rvert}

\newcommand{\eps}{\varepsilon}

\DeclareMathOperator{\Var}{Var}

\allowdisplaybreaks

\renewcommand{\1}{\mathbf{1}}

\begin{document}
 
\title{Multi-Scale CUSUM Tests\\
for Time Dependent Spherical Random
Fields}
\author{Alessia Caponera\footnote{Department of Economics and Finance, LUISS Guido Carli. Email: acaponera@luiss.it}, Domenico Marinucci\footnote{Department of Mathematics, University of Rome Tor Vergata. Email: marinucc@mat.uniroma2.it}, Anna Vidotto\footnote{Department of Mathematics and Applications, University of Naples Federico II. Email: anna.vidotto@unina.it}}
\maketitle

\begin{abstract}
This paper investigates the asymptotic behavior of structural break tests in the harmonic domain for time dependent spherical random fields. In particular, we prove a functional central limit theorem result for the fluctuations over time of the sample spherical harmonic coefficients, under the null of isotropy and stationarity; furthermore, we prove consistency of the corresponding CUSUM test, under a broad range of alternatives, including deterministic trend, abrupt change, and a nontrivial power alternative. Our results are then applied to NCEP data on global temperature: our estimates suggest that Climate Change does not simply affect global average temperatures, but also the nature of spatial fluctuations at different scales.
\end{abstract}

 \tableofcontents

\section{Introduction}

The analysis of time dependent spherical random fields is the natural setting for a number of different areas of applications; some relevant examples include Cosmology, Geophysics and Atmospheric/Climate Sciences, see for instance \cite{CaponeraPanaretos,CaponeraA0S21,Clarke,Porcu18,Porcu} and the references therein for a small sample of recent contributions (more generally, spherical data have been recently considered in different frameworks by \cite{Cheng2019,DiMarzio1,DiMarzio2,Fan,Fan2}, among others). In these areas, it is often a natural question to probe whether structural breaks have occurred over time; the most immediate example of such changes is obviously represented by shifts in the global mean (namely, averaged over space), which would correspond to Global Warming when studying temperature data. Such shifts can be investigated by means of a number of traditional statistical tools, such as the celebrated CUSUM test for structural breaks.

Our purpose in this paper is to exploit harmonic/spectral methods to
investigate multi-scale structural breaks, i.e., modifications of the
statistical model which may go beyond a simple global mean shift. A more rigorous
and complete description of our environment will be given in the sections to
follow; we believe, however, that it is useful to introduce from the start
our motivations by means of a very preliminary analysis on a temperature
data set.

In order to do so, let us first recall that a time dependent spherical
random field is simply a collection of random variables $\{T(x,t), \ (x, t) \in \mathbb{S}^{2} \times \mathbb{Z}\};$ under some regularity
conditions (to be given below) the following spectral
representation holds
\begin{equation*}
T(x,t)=\sum_{\ell =0}^{\infty }T_{\ell }(x,t)=\sum_{\ell = 0}^\infty \sum_{m=-\ell}^{\ell} \beta_{\ell
m}(t)Y_{\ell m}(x)\text{ ,}
\end{equation*}%
where $\left\{ Y_{\ell m}, \ m=-\ell ,\dots,\ell, \ \ell =0,1,2,\dots\right\} $
denotes the set of fully-normalized real spherical harmonics, an orthonormal
basis for the space of square-integrable real-valued functions on the sphere,
see \cite{MaPeCUP}. Heuristically, each term $T_{\ell }(\cdot, t)$ can be viewed as the Fourier
component of the spherical field at time $t$ when projected onto basis elements characterized by
fluctuations of typical scale $\theta _{\ell }\simeq \pi /\ell$. In
particular, the Fourier coefficient corresponding to $\ell =0$ is simply the sample mean at time $t$
computed on the whole sphere, namely
$$
\beta_{00}(t)=\frac1{\sqrt{4\pi}}\int_{\mathbb{S}^2}T(x,t)\,dx;
$$
as such, it is for instance the typical statistic of
interest when discussing global increments of the Earth temperature.
As already mentioned, our aim in this paper is to explore possible changes occurring at different scales; for instance, it could be the case that even when the global mean is unaltered, or in addition to changes in the latter,
other modifications occur in the fluctuations of the variable of interest.

This point is illustrated by a simple, very preliminary data analysis that
we report here just as a motivating example; much greater details will be
given below in Section \ref{sec::ncep}. More precisely, we consider \textit{global (land and ocean) surface temperature anomalies}; the dataset is built starting from the NCEP/NCAR monthly averages of the surface air temperature (in degrees Celsius) from 1948 to 2020, over a global grid with $2.5^{\circ}$ spacing for latitude and longitude, see \cite{ncep}. Following the World Meteorological Organization policy, temperature anomalies are obtained by subtracting the long-term monthly means relative to the 1981-2010 base period. They are then averaged over months to switch from a monthly scale to an annual scale (we refer also to the recent survey \cite{Stein2020} for a general discussion on the role of Statistics for Climate Data; see also \cite{RuizMedina2024} and the references therein for some very recent contributions).

Figure \ref{fig:beta} reports the temporal evolution of $\beta_{00}(t)$, which is nothing else than the sample average of Earth temperature at time $t$; unsurprisingly, this global mean grows steadily from 1949 to 2020.

\begin{figure}[!ht]
    \centering
    \includegraphics[scale=0.7]{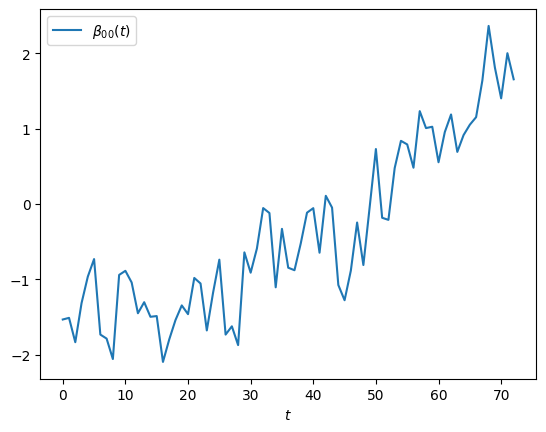}
    \caption{Temporal evolution of $\beta_{0 0}(t)$ from 1948 to 2020;  $t =0$ corresponds to 1948, while $t=72$ corresponds to 2020.}
    \label{fig:beta}
\end{figure}

Let us now consider a plot of the fluctuations of the temperature, with respect to its historical mean (given in Figure \ref{fig:map}). A careful inspection of the colors suggests that not only
some regions exhibit greater changes than others (i.e., the Poles), but also that a small periodic pattern seems to appear, with greater fluctuations that are at an approximated distance of 40-60 degrees across different
latitudes.

\begin{figure}[!ht]
    \centering
    \includegraphics[scale=0.4]{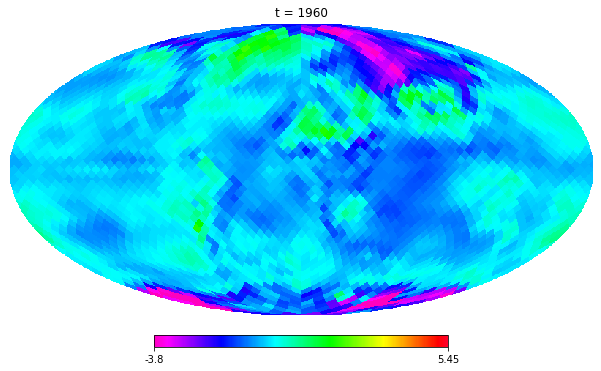}
    \includegraphics[scale=0.4]{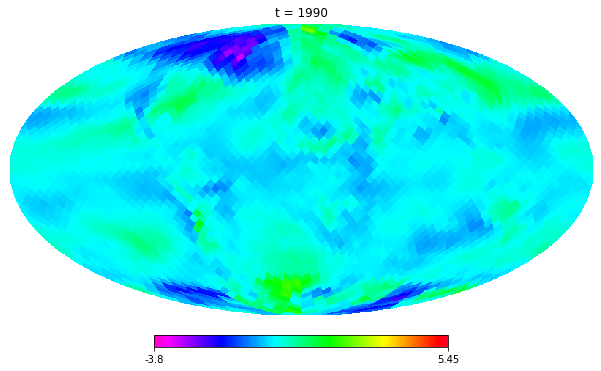}
    \includegraphics[scale=0.4]{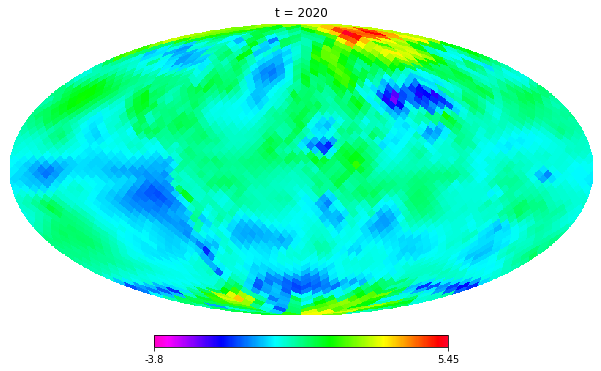}
    \caption{Global surface temperature anomalies (resolution: $3072$ pixels) in years 1960, 1990, and 2020.}
    \label{fig:map}
\end{figure}

This point can be made clearer by introducing the spherical coordinates $\theta \in \lbrack 0,\pi ]$ (co-latitude) and $\phi \in \lbrack 0,2\pi )$ (longitude); in this coordinates, the spectral representation
becomes
\begin{equation*}
T(\theta ,\phi ;t)=\sum_{\ell =0}^{\infty }T_{\ell }(\theta ,\phi
;t)=\sum_{\ell=0}^\infty \sum_{m=-\ell}^\ell \beta_{\ell m}(t)Y_{\ell m}(\theta ,\phi ),
\end{equation*}%
with the analytic expressions%
\begin{eqnarray*}
    Y_{\ell m}(\theta ,\phi ) &=& \begin{dcases}
\sqrt{\frac{2\ell +1}{2\pi }}\sqrt{\frac{(\ell
-|m|)!}{(\ell +|m|)!}}P_{\ell |m|}(\cos \theta )\sin (|m|\phi ) & \text{for } m<0 \\
\sqrt{\frac{2\ell +1}{4\pi }} P_{\ell 0}(\cos \theta ) & \text{for } m=0\\
\sqrt{\frac{2\ell +1}{2\pi }}\sqrt{\frac{(\ell
-m)!}{(\ell +m)!}}P_{\ell m}(\cos \theta )\cos (m\phi ) & \text{for } m > 0 
    \end{dcases},\\
P_{\ell m}(u) &=& \frac{1}{2^\ell \ell !}(1-u^{2})^{m/2}\frac{d^{\ell +m}}{dt^{\ell +m}}(u^{2}-1)^{\ell },  \quad u \in [-1,1], \ m\ge 0.
\end{eqnarray*}%
The $P_{\ell m}$'s are the well-known associated Legendre
functions, whereas the $\beta_{\ell m}(t)$'s can be obtained by means of the Fourier transform
\begin{equation} \label{alm}
\int_{-\pi}^\pi \int_0^\pi T_{\ell }(\theta ,\phi ;t)Y_{\ell m}(\theta ,\phi )\sin\theta d\theta d\phi .
\end{equation}

Note also that, for every $\ell =0,1,2,\dots,$ we have that%
\begin{eqnarray*}
\int_{-\pi }^{\pi }T_{\ell }(\theta ,\phi ;t)d\phi 
%&=&\sqrt{\frac{2\ell +1}{%
%4\pi }}\sqrt{\frac{(\ell -m)!}{(\ell %+m)!}}\sum_{m=-\ell }^{\ell }\beta_{\ell
%m}(t)P_{\ell m}(\cos \theta )\int_{-\pi %}^{\pi }\exp (im\phi )d\phi \\
&=&\sqrt{2\ell +1} \sqrt{\pi} \beta_{\ell 0}(t)P_{\ell 0}(\cos \theta )\text{ .}
\end{eqnarray*}%
In other words, up to a deterministic function, each random
process $\left\{ \beta_{\ell 0}(t), \ t \in \mathbb{Z}\right\}$ captures the temporal evolution of the sample spatial mean for the spectral component corresponding to the
multipole $\ell$, computed at any given latitude. To explore their
behavior, we plot in Figure \ref{fig:beta2} the observed evolution of $\beta_{\ell
0}(t)$ over the time span 1948-2020, for $\ell =2,4,6,8,$. It is remarkable that these
coefficients appear indeed to grow over time, suggesting that further changes in the Earth temperature pattern may have occurred, in addition
to the growth of the global average.

\begin{figure}[!ht]
    \centering
    \includegraphics[scale=0.5]{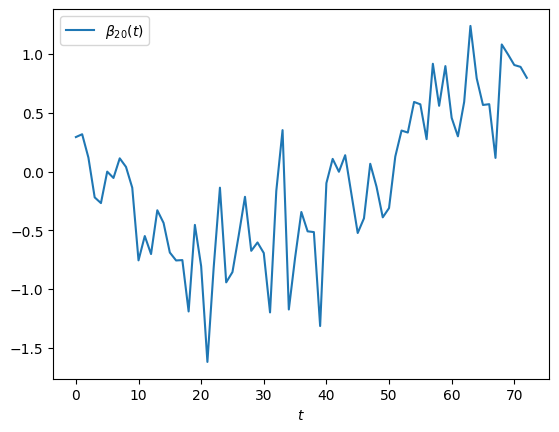}
    \includegraphics[scale=0.5]{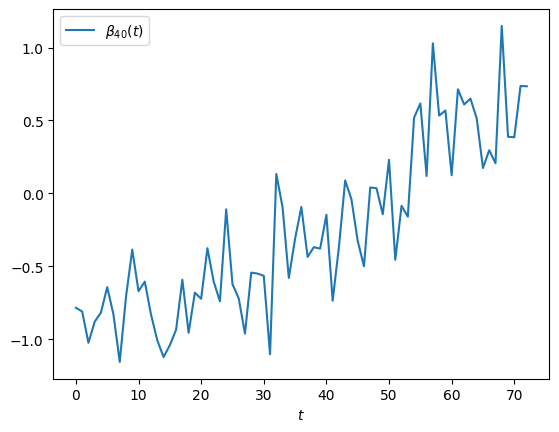}
    \includegraphics[scale=0.5]{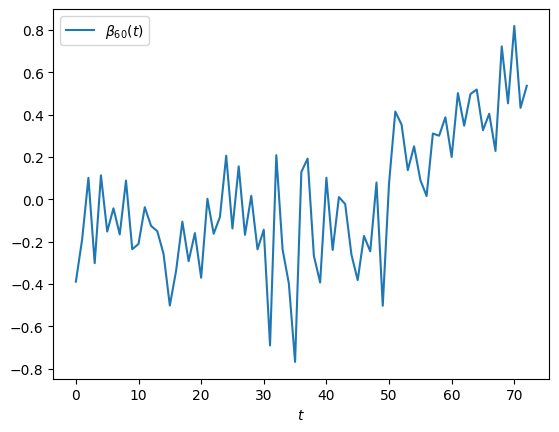}
    \includegraphics[scale=0.5]{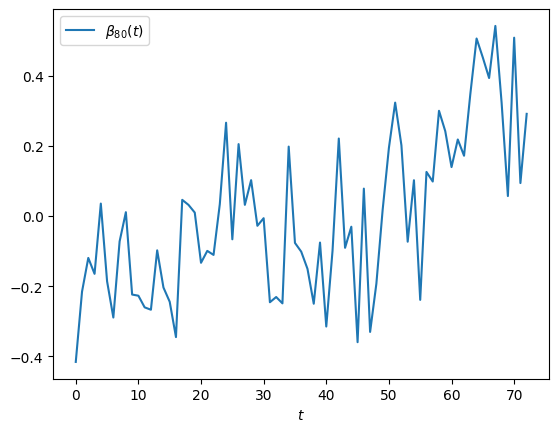}
    \caption{Temporal evolution of $\{\beta_{\ell 0}(t)$, for $\ell=2,4,6,8$, from 1948 to 2020; $t =0$ corresponds to 1948, while $t=72$ corresponds to 2020.}
    \label{fig:beta2}
\end{figure}

This preliminary data exploration suggests that even if we were to subtract the global temperature mean $\beta_{00}(t)$ from each time observation, in order to make it constantly equal to zero, the effects of Climate Change would still be visible in the form of structural changes on different scales. Our purpose in this paper is to devise a class of tests in order to investigate these phenomena more rigorously, which is what we start to do from the next section.

\section{Model}

In this section, we introduce our model of interest. In short, under the
null we assume to deal with a collection of time dependent spherical random fields with a stationary (in time) but anisotropic (in space) mean function; on the other hand, we are going
to probe alternatives that replace the stationary mean by introducing time-varying factors which
may vary across the different multipole components. Moreover, the test statistic will be defined on a fixed window of consecutive frequencies (the so-called \emph{multipoles} in the spherical setting), that is, on $\{\underline{\ell},\underline{\ell}+1,\dots,\overline{\ell}\}$, with $\underline{\ell}, \overline{\ell}$ being two fixed non-negative integer values such that $\underline{\ell}\le\overline{\ell}$ and we define
$$
\triangle{}_{\underline{\ell},\overline{\ell}}:=\{(\ell,m):\ell=\underline{\ell},\dots,\overline{\ell}, \, m=-\ell,\dots,\ell\}.
$$
From now on, we assume that all the constants involved in the $o$- and $O$-notations  and in all upper and lower bounds depend on such multipole window $\{\underline{\ell},\dots,\overline{\ell}\}$.

\subsection{The null hypothesis of stationarity}
Let $\{Z(x,t), \ (x,t) \in \mathbb{S}^2\times \mathbb{Z}\}$ denote a centered strictly isotropic (over space) and strictly stationary (over time) sphere-cross-time random field, that is, a collection of random variables with finite variance and such that%
\begin{eqnarray*}
Z(\cdot,\cdot)&\stackrel{d}{=}&Z(g\,\cdot,\tau+\cdot) \qquad \tau \in \Z\,, \quad g\in SO(3)\,,\\
\mathbb{E}[Z(x,t)] &=& 0\,, \\
\mathbb{E}[Z(x_{1},t_{1})Z(x_{2},t_{2})] &=:& \Gamma (\langle x_{1},x_{2}\rangle,t_{2}-t_{1}),
\end{eqnarray*}
with $\langle \cdot, \cdot\rangle$ being the standard inner product in $\mathbb{R}^3$ and $SO(3)$ is the group of all rotations about the origin of $\R^3$ under the operation of composition.
As anticipated above, it is well-known (see, e.g., \cite{MaPeCUP}) that the space of square-integrable functions on the sphere admits as an orthonormal basis the fully-normalized spherical harmonics $\left\{ Y_{\ell m}, \ m=-\ell, \dots, \ell, \ \ell=0,1,2,\dots\right\}$, which are
eigenfunctions of the spherical Laplacian and hence they satisfy the
Helmholtz equation%
\begin{eqnarray*}
&&\Delta _{\mathbb{S}^{2}}Y_{\ell m} = -\lambda _{\ell }Y_{\ell m}\text{ , }%
\qquad \lambda _{\ell }=\ell (\ell +1)\text{ , } \qquad \ell =1,2,\dots, \\
&&\Delta _{\mathbb{S}^{2}} =\frac{1}{\sin \theta }\frac{\partial }{\partial
\theta }\left (\sin \theta \frac{\partial }{\partial \theta }\right)+\frac{1}{\sin
^{2}\theta }\frac{\partial ^{2}}{\partial \phi ^{2}}\text{ .}
\end{eqnarray*}%
The Spectral Representation Theorem for isotropic random fields ensures that, in $L^{2}(\Omega \times \mathbb{S}^{2})$%,
\begin{equation*}
Z(\theta ,\phi ;t) = \sum_{\ell =0}^{\infty }\sum_{m=-\ell
}^{\ell }a_{\ell m}(t)Y_{\ell m}(\theta ,\phi ).
\end{equation*}%
The idea is now to consider a process $\{T(x,t), \ (x,t) \in \mathbb{S}^2 \times \mathbb{Z}\}$ that satisfies
\begin{equation}\label{H0-eq}
T(x, t) - \mu(x) = Z(x, t), 
\end{equation}
for some deterministic $\mu(\cdot) \in L^2(\mathbb{S}^2)$. The anisotropic (but stationary) mean function then has the $L^{2}$-expansion%
\begin{eqnarray*}
\mu (\theta ,\phi ) &=&\sum_{\ell =0}^{\infty }\sum_{m=-\ell }^{\ell }\mu
_{\ell m}Y_{\ell m}(\theta ,\phi )\text{ ,} \\
\mu _{\ell m}  &=&\int_{-\pi}^\pi \int_0^\pi \mu (\theta ,\phi )Y_{\ell
m}(\theta ,\phi )\sin \theta d\theta d\phi \text{ .}
\end{eqnarray*}%
Hence, we can introduce our model under the null hypothesis in the next assumption.
\begin{Assu}[Null hypothesis $H_0$]\label{Assumption 0}
Under the null hypothesis $H_0$, the field satisfies the equation in \eqref{H0-eq}, and hence it has stationary mean function $\mathbb{E}[T(x,t)]=\mu(x)$ and  spectral representation
\begin{equation*}
T(x,t)=\sum_{\ell=0}^\infty \sum_{m=-\ell}^\ell  a_{\ell m}(t)Y_{\ell m}(x)+\sum_{\ell=0}^\infty \sum_{m=-\ell}^\ell \mu
_{\ell m}Y_{\ell m}(x)\text{ .}
\end{equation*}
\end{Assu}

Under the assumptions on $\{Z(x,t), \ (x,t) \in \mathbb{S}^2 \times \mathbb{Z}\}$ we also have that the array of spherical harmonic coefficients $\{a_{\ell m}(t), \ m=-\ell, \dots, \ell,\ \ell=0,1,2,\dots,\ t \in \mathbb{Z}\}$ is formed by zero-mean uncorrelated (over $\ell$ and $m$) and stationary (over $t$) processes, with covariances
\begin{align*}
C_{\ell}(\tau) &:= \mathbb{E}[a_{\ell m}(t+\tau) a_{\ell m}(t)] \text{ ,} \qquad t,\tau \in \mathbb{Z}\,.
%f_{\ell}(\lambda)&:=\frac{1}{2\pi}\sum_{\tau\in \mathbb{Z}}C_{\ell}(\tau)e^{-i\lambda \tau} \text { , }  \qquad \lambda \in [-\pi,\pi] \text{ .}
\end{align*}
Note that $\{C_\ell(0),\ \ell=0,1,2,\dots\}$ corresponds to the angular power spectrum of the spherical field at a
given time point, for which we will simply write $\{C_\ell,\ \ell=0,1,2,\dots\}$; for simplicity and without loss of generality we assume that $C_\ell(0)>0$ for all $\ell\in \{\underline{\ell},\dots,\overline{\ell}\}$.
We also impose some regularity conditions on the behavior of higher-order cumulants, as detailed in the following subsection.

\subsection{Cumulants and higher-order conditions}

Let us first recall that the joint cumulant of a random vector $(X_1,\dots,X_p)$ is just the Fourier coefficient in the expansion of the logarithm of the joint characteristic function (the expansion is implicitly assumed to exist), see \cite{PeccatiTaqqu}. Of course for Gaussian random variables cumulants of order strictly larger than $2$ are exactly equal to zero; for random spherical harmonics coefficients of isotropic spherical processes, a number of further characterizations for their joint cumulants are discussed in \cite[Chapter 6]{MaPeCUP}. 

In this paper, we are allowing for general non-Gaussian behavior, imposing only the following very broad condition on the joint cumulants:
\begin{Assu}\label{assumption:cum}
For all integers $p \ge 2$ and $(\ell_j,m_j)\in \triangle_{\underline{\ell},\overline{\ell}}$, $j=1,\dots,p$,
\begin{equation}\label{eq:cum}
\sum_{\tau_1, \dots, \tau_{p-1} =-\infty}^{+\infty} \abs{ \operatorname{cum}\big(a_{\ell_1,m_1}(t+\tau_1), \dots, a_{\ell_{p-1},m_{p-1}}(t+\tau_{p-1}), a_{\ell_p,m_p}(t)\big)} < + \infty.
\end{equation}
\end{Assu}

\begin{Rem}
Summability of the cumulants is a standard condition in asymptotic theory of stationary processes (see e.g. \cite{PanaretosTavakoli} and the references therein). For stationary time series it is fulfilled for instance by ARMA processes with i.i.d. innovations and finite moments of all orders. Likewise, by a careful exploitation of the diagram formula, see \cite{MaPeCUP}, subsection 4.3.1, it can be shown that it is fulfilled by polynomial transforms of Gaussian processes with summable covariances.
\end{Rem}

In view of the  strict stationarity condition that we imposed above, the expression \eqref{eq:cum} is strictly invariant with respect to $t\in\Z$. In particular, condition \eqref{eq:cum} (for $p=2$) implies the continuity of the spectral densities of the spherical harmonic coefficients; indeed, for any multipole $\ell$, the associated spectral density function is well defined as
\begin{equation}\label{eq:spec-dens}
f_{\ell}(\lambda):=\frac{1}{2\pi}\sum_{\tau\in \mathbb{Z}}C_{\ell}(\tau)e^{-i\lambda \tau} \text { , }  \qquad \lambda \in [-\pi,\pi] \text{ .}
\end{equation}

\begin{Assu}\label{Assumption 1}
For all integers $\ell\in \{\underline{\ell},\dots,\overline{\ell}\}$, the spectral density at the origin is strictly positive, that is, $f_{\ell}(0)>0$. 
\end{Assu}
This is a standard identifiability condition that is fulfilled, for instance, by all stationary invertible ARMA processes.

\subsection{The test statistic}

We start by defining some sample quantities that are needed in order to introduce our CUSUM test statistics, which can also be viewed as a form of Fourier domain testing for stationarity for functional valued time series, see \cite{Aue,Hormann}.

\begin{Def} [Sample harmonic coefficients] For $t=1,2,\dots,N$, the sample spherical harmonic coefficients are defined
by%
\begin{equation*}
\beta _{\ell m}(t):=\int_{\mathbb{S}^{2}}T(x,t)Y_{\ell m}(x)dx\text{ ,} \qquad (\ell,m) \in \triangle_{\underline{\ell},\overline{\ell}}.
%=a_{\ellm}(t)+\mu _{\ell m}
\end{equation*}
\end{Def}

Note that we are implicitly assuming that the $\beta _{\ell m}(t)$'s can be estimated exactly from the observations, i.e., that the integrals defined in (\ref{alm}) can be computed without approximations. 
That is, we are adopting the assumption that the field is fully
observed over the sphere so that it is possible to compute its Fourier
coefficients, which is standard in the functional data analysis context.

\begin{Def}[Sample harmonic averages] The sample harmonic averages are defined as%
\begin{equation*}
\widehat{\mu }_{\ell m}:=\frac{1}{N}\sum_{t=1}^{N}\beta _{\ell m}(t)\text{ ,} \qquad (\ell,m) \in \triangle_{\underline{\ell},\overline{\ell}}.
\end{equation*}
\end{Def}

\begin{Def}[Sample variance] The sample variance is defined as 
\begin{flalign*}
&\widehat\sigma^2=\sum_{\ell=\underline{\ell}}^{\overline{\ell}}(2\ell+1) 2\pi \,\widehat f_\ell(0)\,, \quad \text{which estimates} \quad \sigma^2:=\sum_{\ell=\underline{\ell}}^{\overline{\ell}} (2\ell+1) 2\pi \, f_\ell(0)\,,
\end{flalign*} 
where
\begin{flalign*}
&2\pi\widehat f_\ell(0)= \sum_{\tau=-q_N}^{q_N}\left(1-\frac{\abs{\tau}}{q_N+1}\right)\widehat{C}_\ell(\tau)=\widehat{C}_\ell(0)+2\sum_{\tau=1}^{q_N}\left(1-\frac{\tau}{q_N+1}\right)\widehat{C}_\ell(\tau)\,,\\
&\widehat{C}_\ell(\tau)=\frac1{2\ell+1}\frac1{N-\abs{\tau}} \sum_{m=-\ell}^\ell \sum_{t=1}^{N-\abs{\tau}}(\beta_{\ell m}(t)-\widehat{\mu}_{\ell m})(\beta_{\ell m}(t+\abs{\tau})-\widehat{\mu}_{\ell m})\,.
\end{flalign*}
The bandwidth parameter $q_N$ is assumed to grow slower than $\sqrt{N}$; formally, 
$$
q_N\to \infty \quad \text{as}\quad  N\to \infty \quad \text{and}\quad  q_N=o\left(\sqrt{N}\right).
$$
\end{Def}

\begin{Def}[The test statistic] The test statistic on which we shall focus is
\begin{equation*}
A_N(s):=\begin{dcases}\frac{1}{\sqrt{N}}\frac{1}{\widehat\sigma}%
\sum_{t=1}^{[Ns]}\sum_{\ell =\underline{\ell}}^{\overline{\ell}}%
\sum_{m=-\ell }^{\ell }(\beta _{\ell m}(t)-\widehat{\mu }_{\ell m}),& \qquad s \in [1/N,1]\,,\\
0&\qquad s \in [0,1/N)
\end{dcases}
\end{equation*}%
where $[\cdot]$ denotes the floor function. 
\end{Def}
In the following, when writing $s\in[0,1]$ we always implicitly make the previous distinction.

\subsection{The alternative hypothesis}\label{sec:alternative}

We shall here introduce our model under the alternative; in this case, we allow the deterministic mean function to vary over time, and hence we consider a process $\{T(x,t), \ (x,t) \in \mathbb{S}^2 \times \mathbb{Z}\}$ that satisfies
\begin{equation}\label{H1-eq}
T(x, t) - \mu(x,t) = Z(x, t), 
\end{equation}
for some $\mu(\cdot, t) \in L^2(\mathbb{S}^2)$, $t \in \mathbb{Z}$. 

\begin{Assu}[Alternative hypothesis $H_1$] \label{Alternative}
Under the alternative hypothesis $H_1$, the field satisfies the equation in \eqref{H1-eq}, so that
\begin{equation*}
T(x,t) = \sum_{\ell=0}^\infty \sum_{m=-\ell}^\ell a_{\ell m}(t)Y_{\ell m}(x)+\sum_{\ell=0}^\infty \sum_{m=-\ell}^\ell  \mu_{\ell m}(t)Y_{\ell m}(x) \text{ ,}
\end{equation*}
with 
\begin{equation*}
    \mu_{\ell m}(t)=\mu_{\ell m;0}+N^{\alpha_\ell}  \,g_{\ell m}\left(\frac{t}{N}\right) , \qquad \alpha_\ell \in \R\,,
\end{equation*}
where $g_{\ell m}$ are bounded piecewise continuous functions on $[0,1]$. In addition, denoting with $$\overline{\alpha}:=\max_{\ell=\underline{\ell},\dots,\overline{\ell}} \alpha_\ell    \,,$$ 
we assume that
$$
\overline{\alpha}\ge-1/2
, \qquad 
\sum_{\ell\in \overline{\mathcal{I}}}\sum_{m=-\ell}^\ell \Var(g_{\ell m}(U)\ne 0\,, 
$$
$\overline{\mathcal{I}}$ being the set of indexes $\{ \ell\in \{\underline{\ell},\dots,\overline{\ell}\}: \alpha_{\ell}= \overline{\alpha} \}$ and $U\sim \operatorname{Unif}(0,1)$, and there exists $s\in(0,1)$ such that 
\begin{equation}\label{cond:int}
\sum_{\ell\in \overline{\mathcal{I}}}\sum_{m=-\ell}^\ell  \left (\int_0^s g_{\ell m}(t) dt - s\int_0^1 g _{\ell m}(u)du\right )\ne0\,.
\end{equation}
\end{Assu}
This model covers a number of nonparametric alternatives and allows to distinguish two possible regimes: \emph{(I)} the \emph{local alternative} regime, which occurs when $\overline{\alpha} = -1/2$, and  \emph{(II)} the \emph{globally consistent} regime when $\overline{\alpha} > -1/2$.
The first case will lead to a nontrivial power for the test, resulting from a mean shift of the 
asymptotic distribution of $A_N(s)$ compared to the one obtained under the null hypothesis $H_0$ (for a similar result in the linear regression setting, see for instance \cite[Theorem 2]{KP:90}); as usual, the power of the test will depend on the size of the shift. 
In the second case we are going to obtain a consistency result for the test; specifically, the convergence to unity of the  test power, with rates that depend on $\overline{\alpha}$ and $q_N$: see Theorem \ref{thm-H1} together with the detailed discussion in Section \ref{sec:3scenarios}.

\section{Main Results}

In this section, we state our main results, i.e., the asymptotic behavior of the test statistic under the null and alternative hypothesis. In the sequel, Assumptions \ref{assumption:cum} and \ref{Assumption 1} will be always taken to hold.
To characterize the limiting distribution of our statistic of interest under the null, we recall a well-known definition.

\begin{Def} A \emph{Brownian bridge} is a zero-mean Gaussian process  $W:[0,1] \rightarrow \mathbb{R}$ with covariance function
\[
\mathbb{E}[W(s)W(s')]=(s\wedge s')-ss'.
\]
\end{Def}

%We are now in the position to state our two main results, which establish the behavior of our test statistic under the null and the alternative assumptions, respectively.

\begin{theorem}\label{thm-H0}
Under the null hypothesis $H_0$, as $N\rightarrow \infty $, 
\begin{equation*}
A_N(s)\Longrightarrow W(s)\text{ ,}
\end{equation*}
where $\Longrightarrow$ denotes as usual weak convergence in the Skorohod space $D[0,1]$.
\end{theorem}

 The previous result ensures the weak convergence under the null of our test statistic; threshold values for the excursion of Kolmogorov-Smirnov or Cram\'er-Von Mises statistics can hence be derived by analytic computations or simulations. Our next result 
 studies the asymptotic behavior of the test statistic under $H_1$ in the two regimes anticipated in Section \ref{sec:alternative}. 
 
 %ensures that the corresponding test are asymptotically consistent, meaning as usual that the probability of Type II errors converges to zero as the number of observations grow.

\begin{theorem}\label{thm-H1}
Under the alternative hypothesis $H_1$, as  $N\rightarrow \infty $, we have 
\begin{itemize}
%\item[(I)] if  $\overline{\alpha}<-1/2$,  $A_N(s)\implies W(s)$;
\item[(I)] if  $\overline{\alpha}=-1/2$, 
$$
A_N(s)\implies W(s)+\frac{1}{\sigma}
\sum_{\ell\in \overline{\mathcal{I}}}\sum_{m=-\ell}^\ell  \left (\int_0^s g_{\ell m}(t) dt - s\int_0^1 g _{\ell m}(u)du\right ),
$$
where $\overline{\mathcal I}=\{\ell \in \{\underline{\ell},\dots, \overline{\ell}\}: \, \alpha_\ell=\overline{\alpha}\}$;
\item[(II)] if  $\overline{\alpha}>-1/2$, there exists $K>0$ such that, taking
$$
r_N(\overline{\alpha})=\begin{dcases}
     N^{\overline{\alpha}+1/2}& \text{when} \quad \overline{\alpha}\in (-1/2,0) \quad \text{and} \quad q_NN^{2\overline{\alpha}} \to c\ge 0\\
     \sqrt{ N/q_N}&\text{when} \quad \overline{\alpha}\in (-1/2,0) \quad \text{and} \quad q_NN^{2\overline{\alpha}} \to +\infty\\
 \sqrt{N/q_N}&\text{when} \quad \overline{\alpha}\ge 0
\end{dcases}\,,
$$
we have that
$$\Prob\left(\sup_{s\in[0,1]}\abs{A_N(s)}>K \cdot r_N(\overline{\alpha}) \right)\to 1\,.$$
\end{itemize}
\end{theorem}

The proofs of Theorems \ref{thm-H0} and \ref{thm-H1} are given in Appendices \ref{proof:thm-H0} and \ref{proof:thm-H1}, respectively. 

\begin{Rem}\label{cor-H1}
For case \emph{(II)} in Theorem \ref{thm-H1}, as  $N\rightarrow \infty $, we can take more explicitly
\begin{itemize}
    \item $q_N=[N^\beta]$ with $0<\beta<1/2$, and
$$
r_N(\overline{\alpha})=\begin{dcases}
     N^{\overline{\alpha}+1/2}& \text{if} \quad \overline{\alpha}\in (-1/2, -\beta/2] \\
     N^{(1-\beta)/2}&\text{if} \quad \overline{\alpha} > -\beta/2
\end{dcases}\,;
$$
\item $q_N=[\log N]$, and
$$
r_N(\overline{\alpha})=\begin{dcases}
     N^{\overline{\alpha}+1/2}& \text{if} \quad \overline{\alpha}\in (-1/2, 0) \\
     \sqrt{ N/\log N}&\text{if} \quad \overline{\alpha} \ge 0
\end{dcases}\,.
$$
\end{itemize}
Note that when $q_N=[N^\beta]$ with $0<\beta<1/2$, we have a phase transition for the rate of the power in $-\beta/2$, while when $q_N$ grows logarithmically in $N$, the phase transition happens in $0$.
In the literature on CUSUM tests, a standard choice for $q_N$ is $[N^{1/3}]$ -- see, for instance, \cite{RGLA:11} and the references therein.
\end{Rem}

\subsection{Three possible scenarios}\label{sec:3scenarios}

In the following, we first analyze two possible specific models for the globally consistent regime, that is case \emph{(II)} of Theorem \ref{thm-H1},  and then we study the power of the test in the case $\overline{\alpha}=1/2$, that is the locally alternative regime, i.e., case \emph{(I)} of Theorem \ref{thm-H1}.

We start  by first recalling that, under Assumption \ref{Alternative},
$$
\mu_{\ell m}(t)=\mu_{\ell m;0}+N^{\alpha_\ell}  \,g_{\ell m}\left(\frac{t}{N}\right).
$$

\paragraph{Deterministic trend.} 
In the first scenario we consider a growing algebraic trend (to the leading order) which can be different from multipole to multipole and can be assumed arbitrarily small:
$$
\alpha_\ell>0 \quad \forall \ell \in \{\underline{\ell},\dots, \overline{\ell}\}\quad \text{and} \quad   g_{\ell m}(u)=\mu_{\ell m;1} u^{\alpha_\ell}\,  .
$$
This model can be even more generilzed by adding further power terms or oscillating components as a remainder, e.g.,
$$
\mu_{\ell m}(t)=\mu_{\ell m;0} + \mu_{\ell m;1} t^{\alpha_\ell} + \mu_{\ell m;2}(t) \, , \qquad \text{with } \lim_{t \rightarrow \infty}\frac{\mu_{\ell m;2}(t)}{t^{\alpha_{\ell}-\epsilon}}=0\, , \text{ for some }\epsilon > 0 ;
$$
these conditions are for instance satisfied when $\mu_{\ell m;2}(t) = t^{\alpha_\ell - \epsilon^*} \log t^k$, for some $\epsilon^*>0$ and some $k \in \mathbb{R}$.
%We are back to the model under the null either by taking $\mu_{\ell m;i}\equiv 0$, $i=1,2$, or by taking $\overline{\alpha}=0$ and $\mu_{\ell m}\equiv \mu_{\ell m;0} + \mu_{\ell m;1}$, for all values of $\ell$ and $m$. 
In this scenario, we are in case \emph{(II)} of Theorem \ref{thm-H1} with $r_N(\overline{\alpha})=\sqrt{N/q_N}$ and 
\begin{align*}
K 
 %&= \frac{\displaystyle \sum_{\ell\in \overline{\mathcal{I}}}\sum_{m=-\ell}^\ell  \left (\int_0^{s^\star} g_{\ell m}(t) dt - s^\star\int_0^1 g _{\ell m}(u)du\right )}{2 \displaystyle \left(\sum_{\ell\in \overline{\mathcal{I}}}\sum_{m=-\ell}^\ell \Var(g_{\ell m}(U)\right)^{1/2}}\\
% &=\frac{\displaystyle \sum_{\ell\in \overline{\mathcal{I}}}\sum_{m=-\ell}^\ell  \mu_{\ell m;1}\left (\int_0^{s^\star} t^{\overline{\alpha}} dt - s^\star\int_0^1 u^{\overline{\alpha}} du\right )}{2 \displaystyle \left(\sum_{\ell\in \overline{\mathcal{I}}}\sum_{m=-\ell}^\ell \mu_{\ell m;1}^2 \Var(U^{\overline{\alpha}})\right)^{1/2}}\\
&=\frac{ s^\star \abs{(s^\star)^{\overline{\alpha}}-1}\sqrt{2\overline{\alpha}+1}}{2 \overline{\alpha}}\frac{\abs{\sum_{\ell\in \overline{\mathcal{I}}}\sum_{m=-\ell}^\ell  \mu_{\ell m;1}}}{ \left(\sum_{\ell\in \overline{\mathcal{I}}}\sum_{m=-\ell}^\ell \mu_{\ell m;1}^2 \right)^{1/2}}
\, ,\end{align*}
with $s^\star \in (0,1)$, see also see also Equation \eqref{eq:cs} in the proof of Theorem \ref{thm-H1}.
Here, the set $\overline{\mathcal{I}}$ can be viewed as the collection of multipoles which exhibit the strongest non-linear trend; we are imposing that these multipoles have coefficients which do not sum to zero, i.e., $\sum_{\ell\in \overline{\mathcal{I}}}\sum_{m=-\ell}^\ell  \mu_{\ell m;1}\ne0$, otherwise the condition is clearly meaningless.
%\begin{Rem}
%Condition (I) in Assumption \ref{Alternative} is satisfied for instance when $\mu_{\ell m;2}(t) = t^{\alpha_\ell - \epsilon^*} \log t^k$, for some $\epsilon^*>0$ and some $k \in \mathbb{R}$.
%\end{Rem}

\paragraph{Abrupt change.}
It is also natural to consider 
the following scenario, in which an abrupt change occurs at time location $[\eta N]$ with $\eta\in(0,1)$, that is,
$$
\alpha_\ell =0\quad  \forall \ell \in \{\underline{\ell},\dots, \overline{\ell}\}\quad \text{and} \quad g_{\ell m}(u)=\mu_{\ell m;1}  \1_{\{u\ge \eta\}}\,,
$$
see e.g.~\cite{RGLA:11}.
This alternative can be generalized to model the presence of multiple change points; however, we stick to the \emph{at most one change point} setting for simplicity. In this scenario, we are in case \emph{(II)} of Theorem \ref{thm-H1} with $r_N(\overline{\alpha})=\sqrt{N/q_N}$ and   
\begin{align*}
K&= \frac{\abs{(s^\star-\eta)\mathbf{1}_{\{s^\star\ge\eta\}}-s^\star(1-\eta)} \abs{\sum_{\ell\in \overline{\mathcal{I}}}\sum_{m=-\ell}^\ell \mu_{\ell m;1} }}{2 \sqrt{\eta(1-\eta)}\left(\sum_{\ell\in \overline{\mathcal{I}}}\sum_{m=-\ell}^\ell \mu_{\ell m;1}^2 \right)^{1/2}}\,.
\end{align*}
In particular, choosing $s^\star=\eta$,
\begin{align*}
K&= \frac{\sqrt{\eta(1-\eta)}}{2}\frac{\abs{\sum_{\ell\in \overline{\mathcal{I}}}\sum_{m=-\ell}^\ell \mu_{\ell m;1} }}{\left(\sum_{\ell\in \overline{\mathcal{I}}}\sum_{m=-\ell}^\ell \mu_{\ell m;1}^2 \right)^{1/2}}\,.
\end{align*}

\paragraph{Nontrivial power alternative.} Another interesting scenario is the one that yields nontrivial asymptotic power, which corresponds to case \emph{(I)} of Theorem \ref{thm-H1}. In this case, denoting
\begin{equation*}
\operatorname{shift}(s):=\frac{1}{\sigma}
\sum_{\ell\in \overline{\mathcal{I}}}\sum_{m=-\ell}^\ell  \left (\int_0^s g_{\ell m}(t) dt - s\int_0^1 g _{\ell m}(u)du\right)
\end{equation*}
and with $q_{1-\gamma}$ the quantile of order $1-\gamma$ of $\sup_s\abs{W(s)}$, the power of the test can be approximated, for $N$ sufficiently large, by
\begin{align}
\Prob\left(\sup_{s\in[0,1]}\abs{W(s)+\operatorname{shift}(s)}>q_{1-\gamma}\right)
&\ge\Prob\bigg(\abs{W(s^\star)+\operatorname{shift}(s^\star)}>q_{1-\gamma}\bigg)\label{ciao1}\\
&=1-\Prob\bigg(-q_{1-\gamma}-\operatorname{shift}(s^\star)<W(s^\star)<q_{1-\gamma}-\operatorname{shift}(s^\star)\bigg)\notag\\
&=1-\Prob\bigg(Z\in\left[\frac{-q_{1-\gamma}-\operatorname{shift}(s^\star)}{\sqrt{s^\star(1-s^\star)}},\frac{q_{1-\gamma}-\operatorname{shift}(s^\star)}{\sqrt{s^\star(1-s^\star)}}\right]\bigg),\label{ciao2}
\end{align}
for any choice of $s^\star\in (0,1)$. Note that the interval in \eqref{ciao2} is of width $2q_{1-\gamma}/\sqrt{s^\star(1-s^\star)}$, which does not depend on the shift. Heuristically, whenever $\abs{\operatorname{shift}(s^\star)}\gg 1$, this interval is placed on the tails of a standard Gaussian random variable $Z$, and hence the probability in \eqref{ciao1} is close to $1$.\\

In the sections to follow, we illustrate the validity of these results by means of an extensive Monte Carlo study and an application to the NCEP temperature data.

%%%%%%%%%%%%%%%%%%%%%%%%%%%%%%%%%%%%%%%%%%%%%%%%%%%%%%%%%%%%%%%%%%%%%%%%%%%%%%%%%%%%%%%%%%%%%%%%%%%%%%%%%%%%%%%%%%%%%%%%%%%%%%%%%%%%%%%%%%%%%%%%%%%%%%%%%%%%%%%%%%%%%%%%%%%%%%%%%%%%%%%%%%%%%%%%%%%%%%%%%%%%%%%%%

\section{Numerical Results and Applications to NCEP Data}\label{sec::ncep}

In this section, we present some numerical evidence on the performance of the test under the null; we then explore the power of the procedure against different alternatives, and we finally implement our proposed methods on real data on global surface temperature anomalies.

\subsection{Simulations under $H_0$}\label{sec::h0-simu}

We generate $N$ serially correlated Gaussian isotropic random fields on the sphere, through the simulations of their first $L$ harmonic coefficients
$$
\{a_{\ell m}(t), \ m=-\ell, \dots, \ell,\ \ell=0,\dots,L-1, \ t=1,\dots,N\}.
$$
Specifically, we simulate from the following stationary autoregressive model of order 1, independently over $\ell$ and $m$,
$$
a_{\ell m}(t) | a_{\ell m}(t-1) \sim \mathcal{N}\bigg(\! \phi_\ell a_{\ell m}(t-1), \, C_{\ell}(0)\! \bigg),  \qquad a_{\ell m}(t) \sim \mathcal{N}\left(0,\frac{C_{\ell}(0)}{1-\phi_\ell^2} \right),
$$
with autoregressive parameters
$$
\phi_0 = 0.4, \quad \phi_1 = -0.3, \quad \phi_2= 0.2, \quad \phi_3 = -0.1,\quad \phi_\ell = 10 \cdot \ell^{-3}, \ \text{for } \ell = 4, \dots, L-1.
$$
and noise variances
$$
C_0(0) = 2, \quad C_\ell(0) =
\frac{2}{\ell(\ell+1)}, \ \ell =1,\dots,L-1.
$$
Our assumptions are easily seen to be satisfied in this case (see, for instance, \cite{SPHARMA, CaponeraA0S21}). Note that the chosen model for $C_\ell(0)$ would not be acceptable for the full range of multipoles, because the resulting field would not have finite variance. However, this is not an issue here, as we are considering only a finite range of values. If necessary, we can implicitly assume that, for greater $\ell$ values, the sequence decreases at a faster rate.

We then consider the following model
$$
\beta_{\ell m}(t) =   a_{\ell m}  (t) + \mu_{\ell m},
$$
with $$\mu_{00} = 5, \quad \mu_{\ell 0} = -\frac{2}{\ell(\ell+1)} \ \text{for $\ell$ even}, \quad \mu_{\ell m} = 0 \ \text{otherwise}.$$ 

We fix $L=20, \ q_N=[N^{1/3}]$, and we compute the test statistic on a equally-spaced grid on the unit interval where each bin is of length $1/N$, for $N=100, 300, 500$. This is repeated, for each $N$, $B = 2000$ times in order to get the approximate distribution of the sup test (Kolmogorov-Smirnov) statistic $\sup_{s \in [0,1]}|A_N(s)|$, which is shown in Figure \ref{fig:hist_h0} for $N=100$. Table \ref{type1-2k} reports the approximate Type I error probability associated with the selected quantiles, which are
 $q_{0.90} = 1.22$, $q_{0.95} = 1.35$ and $q_{0.99} = 1.60$; it is immediately seen that these empirical error probabilities are extremely close to their nominal values, even for sample sizes of order $N=50$. Moreover, as $N$ increases, the Type one error probability gets even closer to the target values, as expected.

%\begin{figure}[ht!]
%    \centering
%\includegraphics[scale=0.6]{figures/Fig_H0/sett3new-A-realization.png}
%\includegraphics[scale=0.5]{figures/Fig_H0/sett3new-countour.png}
%    \caption{Realization of $A_N$ under Model 2 at $N=100$ ($L=30$).}
%    \label{fig:realization_h0}
%\end{figure}

\begin{figure}[ht!]
\centering
\includegraphics[scale=0.5]{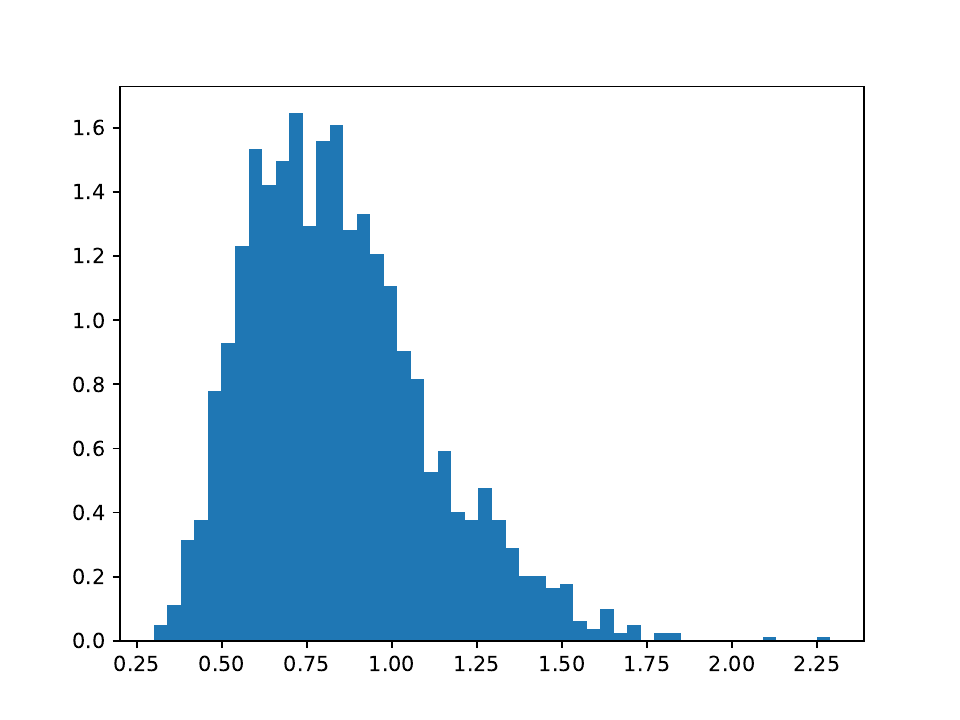}
    \caption{Histogram of the sup test statistic under $H_0$, computed for $L=20$, $N=100$, and based on $B=2000$ replicates.}
   \label{fig:hist_h0}
\end{figure}

\begin{table}[ht!]
    \centering
    \begin{tabular}{c|c|c|c}
     $N$ & $q_{0.90}$ & $q_{0.95}$ & $q_{0.99}$ \\
        \hline
        50 & 0.092& 0.052& 0.012\\
        100 & 0.102& 0.051& 0.0105\\
        500 & 0.0995& 0.05& 0.01
    \end{tabular}
    \caption{Type I error probability corresponding to the selected quantiles of order $1-\gamma = 0.90, 0.95, 0.99$, computed for $L=20$, and based on $B=2000$ replicates.}
    \label{type1-2k}
\end{table}

\subsection{Simulations under $H_1$}

Let us now focus on the alternative hypothesis; under $H_1$, the mean depends on time, i.e.,
$$
\beta_{\ell m}(t) =  a_{\ell m}  (t)+ \mu_{\ell m}(t) .
$$
For the residuals we simulate from the same Gaussian autoregressive model used under $H_0$, while for the mean we consider the \emph{deterministic trend} and \emph{non trivial power alternative} scenarios presented in Section \ref{sec:3scenarios}. For simplicity we set $\mu_{\ell m;0}=0$ uniformly.

Again, all the results are obtained by fixing $q_N=[N^{1/3}]$ and computing the test statistic on a equally-spaced grid on the unit interval where each bin is on length $1/N$, for $B = 2000$ replications.

\paragraph{Deterministic trend.}
As a first model for the non-stationarity mean, we choose
$$\mu_{00}(t) = 5 \cdot t^{\alpha}, \quad \mu_{\ell 0}(t) = -\frac{2}{\ell(\ell+1)} \cdot t^{\alpha} \ \text{for $\ell$ even}, \quad \mu_{\ell m}(t) = 0 \ \text{otherwise},$$ 
which corresponds to the \emph{deterministic trend} scenario with a constant $\alpha$ over the multipoles' window, and coefficients
$$\mu_{00;1} = 5, \quad \mu_{\ell 0;1} = -\frac{2}{\ell(\ell+1)} \ \text{for $\ell$ even}, \quad \mu_{\ell m;1} = 0 \ \text{otherwise}.$$

In particular, we set $L=10$, and we study the behaviour of the power as the number of temporal observations increases, for $\alpha=0.2, 0.5, 1$.
The approximate test's power obtained for different values of $\alpha$ is reported in Figures \ref{fig:power}. The three different curves in each plot correspond to the three different quantiles respectively of order $1-\gamma=0.90,0.95,0.99$.
 Of course for higher values of $\alpha$ the deterministic trends under the alternative are stronger and hence the probability of rejection under $H_1$ becomes very close to one even at rather small sample sizes. In any case the consistency of the testing procedures which was established by our theoretical results shows up clearly in the simulations, as demonstrated by the uniform convergence to unity of the rejection probabilities under all settings, as $N$ diverges.

\begin{figure}[ht!]
\centering
\includegraphics[scale=0.5]{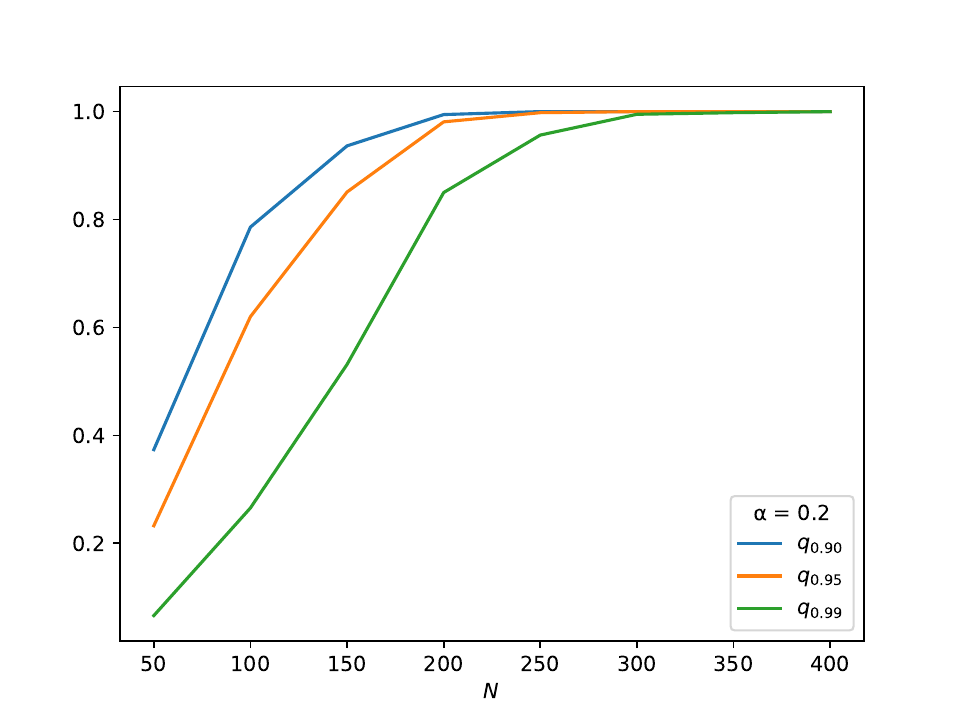}\\
\includegraphics[scale=0.5]{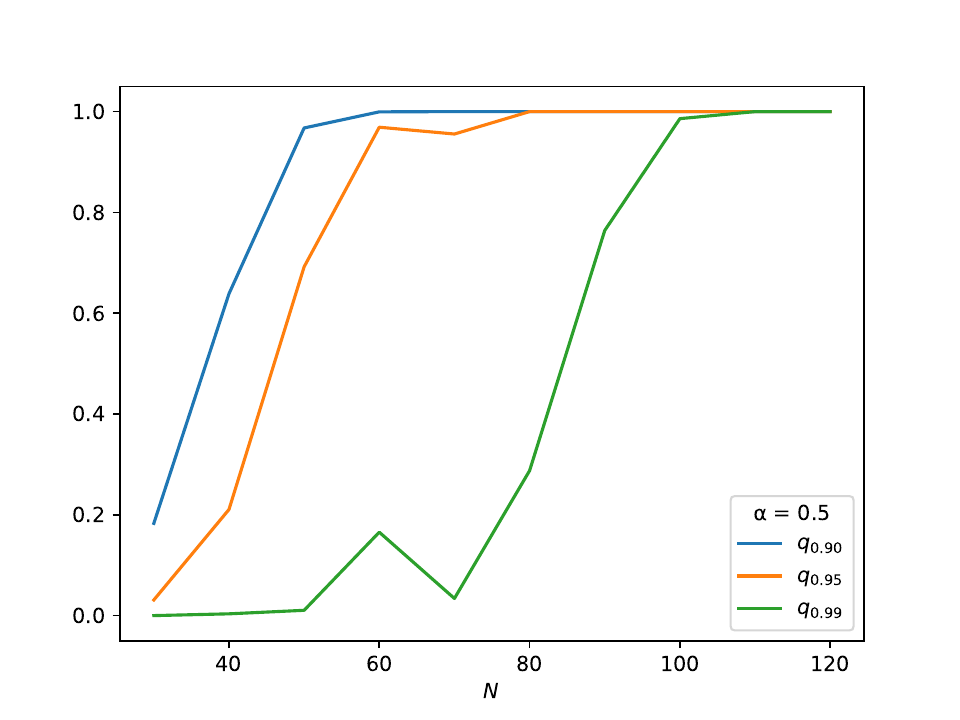}\\
\includegraphics[scale=0.5]{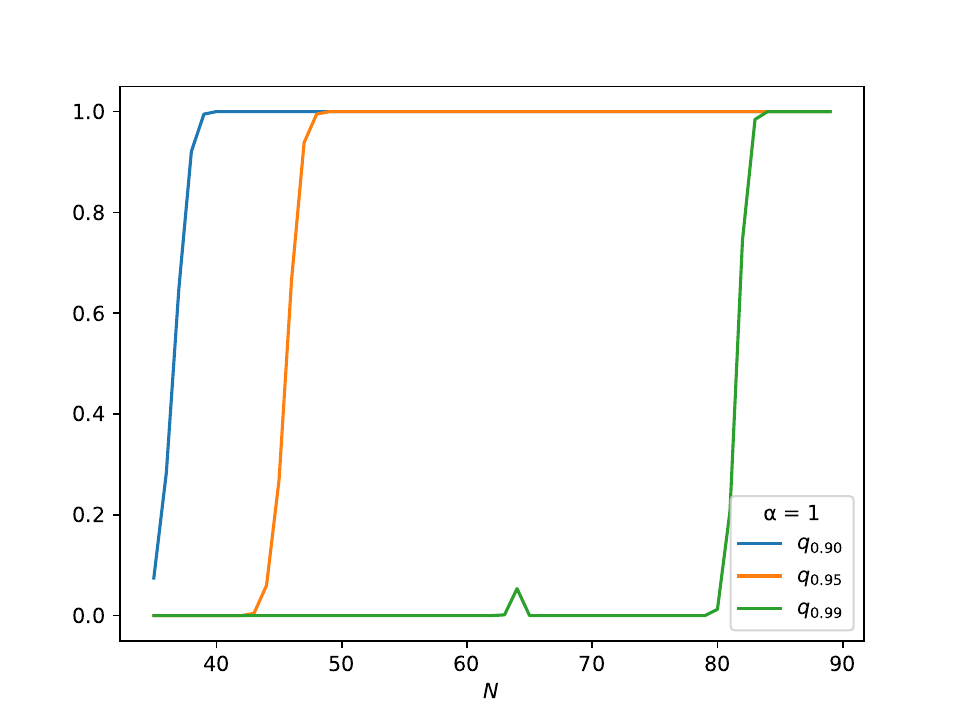}\\
\caption{From top to bottom: test's power under the \emph{deterministic trend} scenario for $\alpha=0.2,0.5,1$, computed for $L=10$, and based on $B=2000$ replicates. The lines correspond to the selected quantiles of order $0.90$ (blu line), $0.95$ (orange line), $0.99$ (green line).}
\label{fig:power}
\end{figure}

\paragraph{Nontrivial power alternative.}
Finally we explore the boundary case alternative. More precisely, we consider the following model for the non-stationary mean
$$\mu_{00}(t) = 5 \cdot\frac{t^\beta}{N^{\beta+1/2}}, \quad \mu_{\ell 0}(t) = -\frac{c}{\ell(\ell+1)} \cdot \frac{t^\beta}{N^{\beta+1/2}} \ \text{for $\ell$ even}, \quad \mu_{\ell m}(t) = 0 \ \text{otherwise},$$ 
for some $c, \beta >0$, which corresponds to the \emph{nontrivial power alternative} scenario with a constant $\alpha = -1/2$ over the mulitpoles' window, $g_{\ell m}(u) = u^\beta$, and coefficients
$$\mu_{00;1} = 5, \quad \mu_{\ell 0;1} = -\frac{2}{\ell(\ell+1)} \ \text{for $\ell$ even}, \quad \mu_{\ell m;1} = 0 \ \text{otherwise}.$$
Here, we set $L=10$, $N=100$, and $\beta = 1$, and we study the behaviour of the power as $\left |\sum_{\ell, m} \mu_{\ell m;1} \right|$ increases, which is obtained by increasing $c$. Indeed, we set $c=2, 100, 200, 300, 400$. Figure \ref{fig:power2} shows the results for the three different quantiles of order $1-\gamma = 0.90,0.95,0.99$.

\begin{figure}[ht!]
\centering
\includegraphics[scale=0.5]{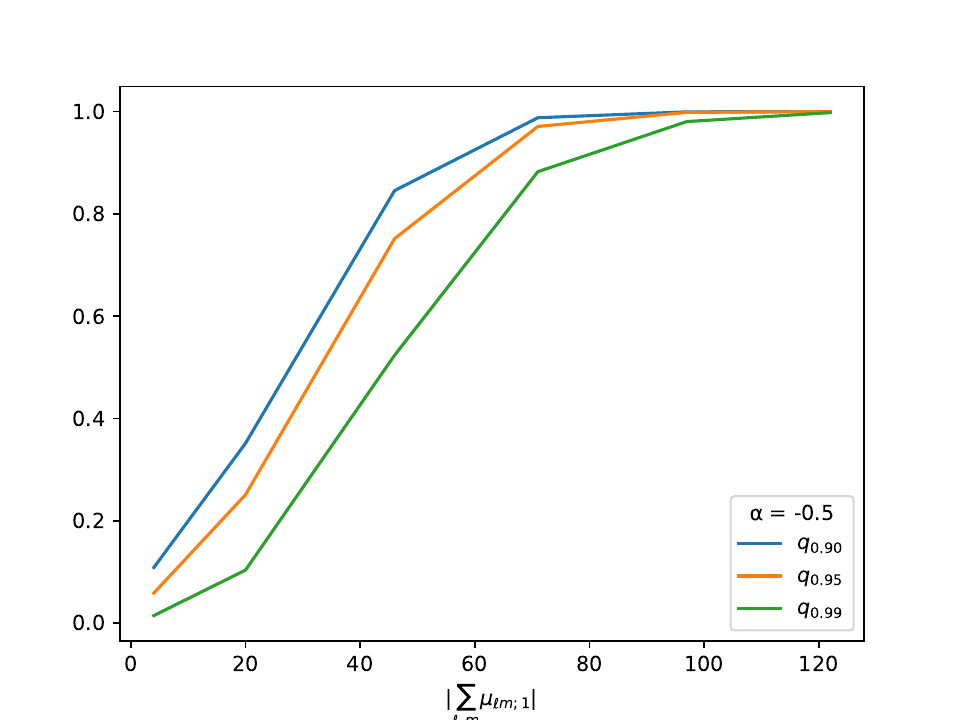}
\caption{Test's power under the \emph{nontrivial power alternative} scenario, computed for $L=10$, and based on $B=2000$ replicates. The lines correspond to the selected quantiles of order $0.90$ (blu line), $0.95$ (orange line), $0.99$ (green line). }
\label{fig:power2}
\end{figure}

\subsection{Application to NCEP data}

As mentioned in the Introduction, the above methodology is applied to \textit{global (land and ocean) surface temperature anomalies}. More in detail, the dataset is built starting from the NCEP/NCAR monthly averages of the surface air temperature (in degrees Celsius) from 1948 to 2020 ($N=73$), over a global grid with $2.5^{\circ}$ spacing for latitude and longitude, see \cite{ncep}. We recall that, following the World Meteorological Organization policy, temperature anomalies are obtained by subtracting the long-term monthly means relative to the 1981-2010 base period; they are then averaged over months to switch from a monthly scale to an annual scale.

Using the \texttt{HEALPix} package (see \cite{HealPix} and the official \texttt{HEALPix} \href{https://healpix.sourceforge.io/}{\underline{website}}), we convert the gridded data into spherical maps with a resolution of $12\cdot\operatorname{NSIDE}^2$ pixels (NSIDE = 16) -- see Figure \ref{fig:map}. Then, we compute the Fourier coefficients and the test statistic based on the window of consecutive multipoles $\{0,1,\dots,12\}$, and $q_N=[N^{1/3}]=4$.

The observed value of the sup test statistic is about 3.27 (see also Figure \ref{fig:ncep-test}), which strongly suggests rejecting the null hypothesis of a stationary (in time) mean function.

\begin{figure}[ht!]
    \centering
    \includegraphics[scale=0.6]{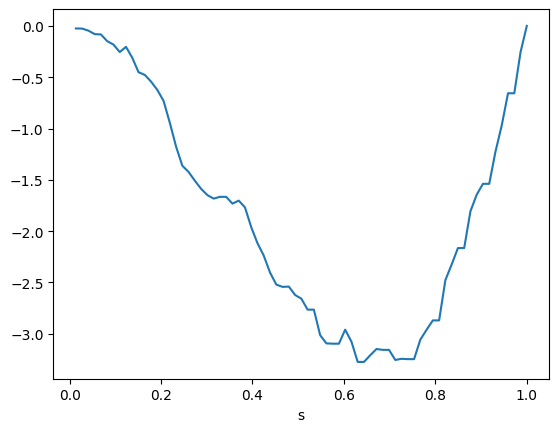}
    \caption{Observed sup test statistic computed over the window of consecutive multipoles $\{0,1,\dots,12\}$.}
    \label{fig:ncep-test}
\end{figure}

We now run the same test, neglecting the multipole $\ell=0$, that is, subtracting at each time $t$ the global sample average of the temperature, which is then made constant over time. We then repeat the same approach, dropping in turn the first $\underline{\ell}-1$ multipoles, for $\underline{\ell}=2,4,6,8$: this way we explore the existence of non-stationarity features that do not involve the mean or the following lowest multipoles. Our test statistic is still significant for $\gamma = 0.01$ (see Table \ref{tab:ncep-test}) in each of these cases (apart from the very last, where significance occurs only at the level $\gamma = 0.05$). These results strongly suggest that climate change occurs not only in the mean temperature at a global level but also in the nature of fluctuations on a number of different scales corresponding to a few dozen degrees. The meaning and consequences of these preliminary findings on the nature of non-stationarity for climate data are left for future investigations.

\begin{table}[ht!]
    \centering
    \begin{tabular}{c|c|c|c|c|c|c}
        $\underline{\ell}$ & 0 & 1 & 2 & 4 & 6 & 8 \\
        \hline
        sup& 3.27& 2.94& 2.72& 2.17& 1.83& 1.48
    \end{tabular}
    \caption{Observed sup test statistic for the NCEP data, computed over the window of consecutive multipoles $\{\underline{\ell},\underline{\ell}+1, \dots, 12\}$.}
    \label{tab:ncep-test}
\end{table}

%%%%%%%%%%%%%%%%%%%%%%%%%%%%%%%%%%%%%%%%%%%%%%%%%%%%%%%%%%%%%%%%%%%%%%%%%%%%%%%%%%%%%%%%%%%%%%%%%%%%%%%%%%%%%%%%%%%%%%%%%%%%%%%%%%%%%%%%%%%%%%%%%%%

\appendix

\section{On the Consistency of $\widehat\sigma^2$}

\subsection{Convergence under the null hypothesis $H_0$}

The next Lemma gives the rate of convergence of $\widehat\sigma$ to $\sigma$, under the null $H_0$.

\begin{Lem}\label{lemma:sigma-H0}
Under the null hypothesis $H_0$, as $N\to \infty$,
$$
\mathbb{E}\left |\widehat\sigma - \sigma\right| = O\left (\frac{q_N}{\sqrt{N}}\right ).
$$
\end{Lem}
\begin{proof}[Proof of Lemma \ref{lemma:sigma-H0}]
In the following we will make extensive use of Assumption \ref{assumption:cum}.

By standard arguments, we can write
\begin{flalign*}
\widehat\sigma^2-\sigma^2&=\sum_{\ell=\underline{\ell}}^{\overline{\ell}} (2\ell+1)\sum_{\tau=-q_N}^{q_N} \left(1-\frac{\abs{\tau}}{q_N+1}\right)\left[\widehat C_\ell(\tau)-C_\ell(\tau)\right]\\
&\quad + \sum_{\ell=\underline{\ell}}^{\overline{\ell}} (2\ell+1)\sum_{\tau=-q_N}^{q_N}\frac{\abs{\tau}}{q_N+1}C_\ell(\tau) + 2 \, \sum_{\ell=\underline{\ell}}^{\overline{\ell}} (2\ell+1)\sum_{\tau=q_N+1}^{\infty}C_\ell(\tau)\\
&=\sum_{\ell=\underline{\ell}}^{\overline{\ell}} (2\ell+1)\sum_{\tau=-q_N}^{q_N} \left(1-\frac{\abs{\tau}}{q_N+1}\right)\left[\widehat C_\ell(\tau)-C_\ell(\tau)\right]+o_\Prob(1)\,.
\end{flalign*}
Heuristically our argument can be illustrated as follows: we will show that the expectation of the terms $\left[\widehat C_\ell(\tau)-C_\ell(\tau)\right]$ is uniformly of order $O(\frac{1}{\sqrt{N}})$, whence the result will follow immediately because the first sum is finite and the second has clearly cardinality $O(q_N)$. Indeed, 
\begin{flalign*}
&\E{\abs{\sum_{\ell=\underline{\ell}}^{\overline{\ell}} (2\ell+1)\sum_{\tau=-q_N}^{q_N} \left(1-\frac{\abs{\tau}}{q_N+1}\right)\left[\widehat C_\ell(\tau)-C_\ell(\tau)\right]}}\\
&\le \sum_{\ell=\underline{\ell}}^{\overline{\ell}} (2\ell+1)\sum_{\tau=-q_N}^{q_N}\E{\abs{\widehat C_\ell(\tau)-C_\ell(\tau)}}\\
&=\sum_{\ell=\underline{\ell}}^{\overline{\ell}} (2\ell+1)\sum_{\tau=-q_N}^{q_N}\E{\abs{\frac{1}{2\ell+1}\frac1{N-\abs{\tau}}\sum_{m=-\ell}^\ell \sum_{t=1}^{N-\abs{\tau}} (a_{\ell m}(t)-\overline{a}_{\ell m}(N))(a_{\ell m}(t+\abs{\tau})-\overline{a}_{\ell m}(N))-C_\ell(\tau)}}\\
&\le\sum_{\ell=\underline{\ell}}^{\overline{\ell}}\sum_{m=-\ell}^\ell\sum_{\tau=-q_N}^{q_N}\E{\abs{\frac1{N-\abs{\tau}}\sum_{t=1}^{N-\abs{\tau}} (a_{\ell m}(t)-\overline{a}_{\ell m}(N))(a_{\ell m}(t+\abs{\tau})-\overline{a}_{\ell m}(N))-C_\ell(\tau)}}\\
&\le\sum_{\ell=\underline{\ell}}^{\overline{\ell}}\sum_{m=-\ell}^\ell\sum_{\tau=-q_N}^{q_N}\sqrt{\E{\left(\frac1{N-\abs{\tau}}\sum_{t=1}^{N-\abs{\tau}} a_{\ell m}(t)a_{\ell m}(t+\abs{\tau})-C_\ell(\tau)\right)^2}}\\
&+\sum_{\ell=\underline{\ell}}^{\overline{\ell}}\sum_{m=-\ell}^\ell\sum_{\tau=-q_N}^{q_N}\left (\E{ \overline{a}^2_{\ell m}(N )} + 2\sqrt{\E{ \overline{a}^2_{\ell m}(N-\abs{\tau})}} \sqrt{\E{ \overline{a}^2_{\ell m}(N)}} \right), 
\end{flalign*}
and it can be easily seen that
\begin{flalign*}
\E{ \overline{a}^2_{\ell m}(N-\abs{\tau})} = \E{ \left(\frac{1}{N-\abs{\tau}} \sum_{t=1}^{N-\abs{\tau}} a_{\ell m}(t)\right)^2} \le \frac{1}{N-\abs{\tau}} \sum_{v=-\infty}^{+\infty} |C_\ell(v)|, \qquad \text{for } \tau = 0, 1, \dots, q_N.
\end{flalign*}
Moreover,
\begin{align*}
&\E{\left(\frac1{N-\abs{\tau}}\sum_{t=1}^{N-\abs{\tau}} a_{\ell m}(t)a_{\ell m}(t+\abs{\tau})-C_\ell(\tau)\right)^2} \\
%&= \E{\left(\frac1{N-\abs{\tau}}\sum_{t=1}^{N-\abs{\tau}} a_{\ell m}(t)a_{\ell m}(t+\abs{\tau})\right)^2} - C^2_\ell(\tau) \\
&= \frac{1}{(N-\abs{\tau})^2} \sum_{t,t'=1}^{N-\abs{\tau}} \E{a_{\ell m}(t)a_{\ell m}(t+\abs{\tau})a_{\ell m}(t')a_{\ell m}(t'+\abs{\tau})  }  - C^2_{\ell}(\tau)\\
&=O\left (\frac{1}{N-\abs{\tau}}\right).
\end{align*}
In fact, by stationarity,
\begin{align*}
&\E{a _{\ell m}(t_1)a _{\ell m}(t_1+\abs{\tau})a _{\ell m}(t_2)a _{\ell m}(t_2+\abs{\tau})}\\
&=\E{a _{\ell m}(t_1)a _{\ell m}(t_1+\abs{\tau})}\E{a _{\ell m}(t_2)a _{\ell m}(t_2+\abs{\tau})}+\E{a _{\ell m}(t_1)a _{\ell m}(t_2)}\E{a _{\ell m}(t_1+\abs{\tau})a _{\ell m}(t_2+\abs{\tau})}\\
&+\E{a _{\ell m}(t_1)a _{\ell m}(t_2+\abs{\tau})}\E{a _{\ell m}(t_1+\abs{\tau})a _{\ell m}(t_2)}+\operatorname{cum}\big(a _{\ell m}(t_1),a _{\ell m}(t_1+\abs{\tau}),a _{\ell m}(t_2),a _{\ell m}(t_2+\abs{\tau})\big)\\
&=C^2_\ell(\tau)+C^2_{\ell}(t_2-t_1)+C_\ell (t_2-t_1 +\abs{\tau} )C_\ell(t_1-t_2+\abs{\tau})\\&+\operatorname{cum}\big(a _{\ell m}(0),a _{\ell m}(\abs{\tau}),a _{\ell m}(t_2-t_1),a _{\ell m}(t_2-t_1+\abs{\tau})\big),
\end{align*}
and hence
\begin{flalign*}
&\E{\left(\frac1{N-\abs{\tau}}\sum_{t=1}^{N-\abs{\tau}} a_{\ell m}(t)a_{\ell m}(t+\abs{\tau})-C_\ell(\tau)\right)^2} \\
&= \frac1{(N-\abs{\tau})^2} \sum_{t_1,t_2=1}^{{N-\abs{\tau}}} 
\bigg[C^2_{\ell}(t_2-t_1)+C_\ell (t_2-t_1 +\abs{\tau} )C_\ell(t_1-t_2+\abs{\tau})\\
&+\operatorname{cum}\big(a _{\ell m}(0),a _{\ell m}(\abs{\tau}),a _{\ell m}(t_2-t_1),a _{\ell m}(t_2-t_1+\abs{\tau})\big)\bigg]\\
&\le\frac{1}{(N-\abs{\tau})}\sum_{v=-\infty}^{\infty} 
\bigg[C^2_{\ell}(v)+\abs{C_\ell (v +\abs{\tau} )C_\ell(-v+\abs{\tau})}\\
&+\abs{\operatorname{cum}\big(a _{\ell m}(0),a _{\ell m}(\abs{\tau}),a _{\ell m}(v),a _{\ell m}(v+\abs{\tau})\big)}\bigg]\\
&\le \frac{2}{(N-\abs{\tau})}\sum_{v=-\infty}^{\infty}  C^2_{\ell}(v)+\frac{1}{(N-\abs{\tau})} \sum_{v_1,v_2,v_3=-\infty}^\infty \abs{\operatorname{cum}\big(a _{\ell m}(0),a _{\ell m}(v_1),a _{\ell m}(v_2),a _{\ell m}(v_3)\big)},
\end{flalign*}
whence the result follows because both series are summable by assumption.
\end{proof}

\subsection{Convergence under the alternative hypothesis $H_1$}

In the following Lemma, we study the behaviour of $\widehat\sigma^2$ under the alternative hypothesis $H_1$.

\begin{Lem}\label{lemma:sigma-hat}
As $N\to\infty$,
\begin{flalign*}
\widehat\sigma^2=\begin{dcases} \sigma^2 + q_N N^{2\overline{\alpha}} \,  \sum_{\ell\in \overline{\mathcal{I}}}\sum_{m=-\ell}^\ell \Var\left(g_{\ell m}(U)\right)+O_\Prob( q_N N^{\overline{\alpha}-1/2})
%o_\Prob\left(q_N N^{2\overline{\alpha}}\right) 
& \text{ if }\quad  \overline{\alpha}>-1/2\\[4pt]
 \sigma^2 +o_\Prob\left(1\right)&  \text{ if }\quad  \overline{\alpha} \le -1/2
\end{dcases},
\end{flalign*}
where $U\sim\operatorname{U}(0,1)$.
\end{Lem}

As a consequence, when $\overline{\alpha} >-1/2$, we can distinguish two different cases:
\begin{itemize}
\item if $q_N N^{2\overline{\alpha}} \to c\ge0$ (without loss of generality), then
$$\widehat \sigma^2
\overset{p}{\to}  \sigma^2 + c\, \sum_{\ell\in \overline{\mathcal{I}}}\sum_{m=-\ell}^\ell \Var\left(g_{\ell m}(U)\right);$$
\item if $q_N N^{2\overline{\alpha}} \to +\infty$, then
$$\frac{\widehat \sigma^2}{q_N N^{2\overline{\alpha}}}
\overset{p}{\to} \sum_{\ell\in \overline{\mathcal{I}}}\sum_{m=-\ell}^\ell \Var\left(g_{\ell m}(U)\right).$$
\end{itemize}

% Recall that $q_N$ is chosen such that $q_N N^{-1/2} \to 0$, as $N\to \infty$.
% Hence, if $q_N=[N^{\beta}]$, with $\beta\in(0,1/2)$, then 
% $$
% \widehat \sigma^2=\sigma^2+ N^{2\overline{\alpha}+\beta} \sum_{\ell\in \overline{\mathcal{I}}}\sum_{m=-\ell}^\ell \Var\left(g_{\ell m}(U)\right) +  o_\Prob( N^{2\overline{\alpha}+\beta})\,, \qquad \text{if } \overline{\alpha} > -1/2.
% $$
% We can distinguish three different regimes:
% \begin{itemize}
% \item if $\overline{\alpha}< -\beta/2$, then $N^{2\overline{\alpha}+\beta} \to 0$, and 
% $\widehat \sigma^2
% \overset{p}{\to} \sigma^2$;
% \item if $\overline{\alpha} = -\beta/2$, then  $N^{2\overline{\alpha}+\beta} = 1$, and 
% $\widehat \sigma^2
% \overset{p}{\to}  \sigma^2 + \sum_{\ell\in \overline{\mathcal{I}}}\sum_{m=-\ell}^\ell \Var\left(g_{\ell m}(U)\right)$;
% \item if $\overline{\alpha}>-\beta/2$, then $N^{2\overline{\alpha}+\beta} \to \infty$, and 
% $$\frac{\widehat \sigma^2}{N^{2\overline{\alpha}+\beta}}
% \overset{p}{\to} \sum_{\ell\in \overline{\mathcal{I}}}\sum_{m=-\ell}^\ell \Var\left(g_{\ell m}(U)\right).$$
% \end{itemize}

% Note also that if $q_N = [\log N]$, the transition happens in 0. 
% %It is reasonable to choose $q_N = [N^\beta]$ for some $\beta \in (0, 1/2)$, and hence from now on we assume it.

\proof[Proof of Lemma \ref{lemma:sigma-hat}]
We have to study the limit as $N\to\infty$ of the sample variance estimator
\begin{flalign*}
\widehat\sigma^2
&=\sum_{\ell=\underline{\ell}}^{\overline{\ell}}\sum_{m=-\ell}^\ell \sum_{\tau=-q_N }^{q_N}\left(1-\frac{\abs{\tau}}{q_N+1}\right)\frac1{N-\abs{\tau}}  \sum_{t=1}^{N-\abs{\tau}}(\beta_{\ell m}(t)-\widehat{\mu}_{\ell m})(\beta_{\ell m}(t+\abs{\tau})-\widehat{\mu}_{\ell m}).
\end{flalign*}
Note that
\begin{align}\label{eq:dec}
&(\beta_{\ell m}(t)-\widehat{\mu}_{\ell m})(\beta_{\ell m}(t+\abs{\tau})-\widehat{\mu}_{\ell m})\notag\\
%&=\left (a_{{\ell} m }(t)-\frac 1N\sum_{u=1}^N a_{{\ell} m}(u)+\mu _{\ell m}(t) -\frac 1N\sum_{u=1}^N\mu _{\ell m}(u)\right )	\notag\\
%&\quad \times\left (a_{{\ell} m }(t+\abs{\tau})-\frac 1N\sum_{u=1}^N a_{{\ell} m}(u)+\mu _{\ell m}(t+\abs{\tau}) -\frac 1N\sum_{u=1}^N\mu _{\ell m}(u)\right )\notag\\
& = \left (a_{{\ell} m }(t)-\frac 1N\sum_{u=1}^N a_{{\ell} m}(u)\right)\left (a_{{\ell} m }(t+\abs{\tau})-\frac 1N\sum_{u=1}^N a_{{\ell} m}(u)\right)\notag \\
&+    \left (\mu _{\ell m}(t) -\frac 1N\sum_{u=1}^N\mu _{\ell m}(u)\right )\left (\mu _{\ell m}(t+\abs{\tau}) -\frac 1N\sum_{u=1}^N\mu _{\ell m}(u)\right )\notag\\
&+ \left (a_{{\ell} m }(t+\abs{\tau})-\frac 1N\sum_{u=1}^N a_{{\ell} m}(u)\right)\left (\mu _{\ell m}(t) -\frac 1N\sum_{u=1}^N\mu _{\ell m}(u)\right )\notag\\
&+ \left (a_{{\ell} m }(t)-\frac 1N\sum_{u=1}^N a_{{\ell} m}(u)\right)\left (\mu _{\ell m}(t+\abs{\tau}) -\frac 1N\sum_{u=1}^N\mu _{\ell m}(u)\right ).
\end{align}

The first term in the last sum has been already studied under $H_0$, and gives the convergence in probability to $\sigma$, as $N\to\infty$.

For the second term, we get 
\begin{align*}
&\sum_{\ell,m}\sum_{\tau=-q_N}^{q_N}\left(1-\frac{\abs{\tau}}{q_N+1}\right)\frac1{N-\abs{\tau}} \sum_{t=1}^{N-\abs{\tau}}\left (\mu _{\ell m}(t) -\frac 1N\sum_{u=1}^N\mu _{\ell m}(u)\right )\left (\mu _{\ell m}(t+\abs{\tau}) -\frac 1N\sum_{u=1}^N\mu _{\ell m}(u)\right )\\
&= q_N N^{2\overline{\alpha}} \sum_{\ell,m}\frac{N^{2\alpha_\ell}}{N^{2\overline{\alpha}}} \frac{1}{ q_N} \sum_{\tau=-q_N}^{q_N}\left(1-\frac{\abs{\tau}}{q_N+1}\right) \\
&\times\frac1{N-\abs{\tau}} \sum_{t=1}^{N-\abs{\tau}}\left (g _{\ell m}\left(\frac{t}{N}\right) -\frac 1N\sum_{u=1}^N g_{\ell m}\left(\frac{u}{N}\right)\right )\left (g _{\ell m}\left(\frac{t+\abs{\tau}}{N}\right) -\frac 1N\sum_{u=1}^N g _{\ell m}\left(\frac{u}{N}\right)\right )\\
&= q_N N^{2\overline{\alpha}} \left(\sum_{\ell\in \overline{\mathcal{I}}} \sum_{m=-\ell}^\ell \int_{-1}^1 (1-\abs{x}) dx \int_0^1 \left(g_{\ell m}(t) -\int_0^1g_{\ell m}(u) \,du\right)^2 dt + o(1) \right)\\
&=   q_N N^{2\overline{\alpha}} \left (\sum_{\ell\in \overline{\mathcal{I}}}\sum_{m=-\ell}^\ell \Var\left(g_{\ell m}(U)\right) + o(1) \right),
\end{align*}
where we used the fact that $q_N=o(\sqrt N)$ and the $g_{\ell m}$ functions are piecewise andcontinuous bounded on $[0,1]$ to exploit their Riemannian integrability.
For the last term (as well as for the third), we obtain
\begin{align*}
&\sum_{\ell,m}\sum_{\tau=-q_N}^{q_N}\left(1-\frac{\abs{\tau}}{q_N+1}\right)\frac1{N-\abs{\tau}} \sum_{t=1}^{N-\abs{\tau}}\left (a_{\ell m}(t) -\frac 1N\sum_{u=1}^N a_{\ell m}(u)\right )\left (\mu_{\ell m}(t+\abs{\tau}) -\frac 1N\sum_{u=1}^N\mu _{\ell m}(u)\right )\\
&= q_N N^{\overline{\alpha}} \sum_{\ell,m}\frac{N^{\alpha_\ell}}{N^{\overline{\alpha}}} \frac{1}{ q_N} \sum_{\tau=-q_N}^{q_N}\left(1-\frac{\abs{\tau}}{q_N+1}\right) \\
&\times \frac1{N-\abs{\tau}} \sum_{t=1}^{N-\abs{\tau}}\left (a_{\ell m}(t) -\frac 1N\sum_{u=1}^N a_{\ell m}(u)\right )\left (g _{\ell m}\left(\frac{t+\abs{\tau}}{N}\right) -\frac 1N\sum_{u=1}^N g _{\ell m}\left(\frac{u}{N}\right)\right )\\
&= O_\Prob(q_N N^{\overline{\alpha} - 1/2}).
\end{align*}
Indeed, we observe that its expectation is zero and
\begin{align*}
&\mathbb{E}\left [\frac1{N-\abs{\tau}} \sum_{t=1}^{N-\abs{\tau}}\left (a_{\ell m}(t) -\frac 1N\sum_{u=1}^N a_{\ell m}(u)\right ) \left (g _{\ell m}\left(\frac{t+\abs{\tau}}{N}\right) -\frac 1N\sum_{u=1}^N g _{\ell m}\left(\frac{u}{N}\right)\right ) \right]^2\\
&= \frac1{(N-\abs{\tau})^2}  \sum_{t,t'=1}^{N-\abs{\tau}} \mathbb{E}\left[\left (a_{\ell m}(t) -\frac 1N\sum_{u=1}^N a_{\ell m}(u)\right )  \left (a_{\ell m}(t') -\frac 1N\sum_{u=1}^N a_{\ell m}(u)\right ) \right]\\
&\times \left (g _{\ell m}\left(\frac{t+\abs{\tau}}{N}\right) -\frac 1N\sum_{u=1}^N g _{\ell m}\left(\frac{u}{N}\right)\right ) \left (g _{\ell m}\left(\frac{t'+\abs{\tau}}{N}\right) -\frac 1N\sum_{u=1}^N g _{\ell m}\left(\frac{u}{N}\right)\right ) \\
&= \frac1{(N-\abs{\tau})^2}  \sum_{t,t'=1}^{N-\abs{\tau}} \left ( C_\ell(t-t') + \frac{1}{N^2} \sum_{u,u'=1}^N C_\ell(u-u') - \frac1N \sum_{u=1}^N C_\ell(t-u) - \frac1N \sum_{u=1}^N C_\ell(t'-u) \right ) \\
&\times \left (g _{\ell m}\left(\frac{t+\abs{\tau}}{N}\right) -\frac 1N\sum_{u=1}^N g _{\ell m}\left(\frac{u}{N}\right)\right ) \left (g _{\ell m}\left(\frac{t'+\abs{\tau}}{N}\right) -\frac 1N\sum_{u=1}^N g _{\ell m}\left(\frac{u}{N}\right)\right ) \\
&= O\left ( \frac1{N-\abs{\tau}}\right),
\end{align*}
where we have used Assumption \ref{assumption:cum} and again the regularity conditions on the functions $g_{\ell m}$'s. By combining all the previous computations, we get
$$
\widehat \sigma^2=\sigma^2+ q_N N^{2\overline{\alpha}} \left(\sum_{\ell\in \overline{\mathcal{I}}}\sum_{m=-\ell}^\ell \Var\left(g_{\ell m}(U)\right) + o(1)\right)+ O_\Prob(q_N N^{\overline{\alpha}-1/2}).$$
%If $\overline{\alpha}>-1/2$
Note that if $\overline{\alpha} \le -1/2$, $O_\Prob( q_N N^{\overline{\alpha}-1/2}) = o_\Prob(1)$, and $q_N N^{2\overline{\alpha}} = o(1)$. 
%On the other hand, if $\overline{\alpha} > - 1/2$, $O_\Prob( q_N N^{\overline{\alpha}-1/2}) = o_\Prob(q_N N^{\overline{2\alpha}})$. 
Hence, the claimed result is established.
\endproof

%%%%%%%%%%%%%%%%%%%%%%%%%%%%%%%%%%%%%%%%%%%%%%%%%%%%%%%%%%%%%%%%%%%%%%%%%%%%%%%%%%%%%%%%%%%%%%%%%%%%%%%%%%%%%%%%%%%%%%%

\section{Proof of Theorem \ref{thm-H0}}\label{proof:thm-H0}

The main idea in the proof is to split our test statistic into two different parts; in the first, the normalization is implemented by means of the (unknown) normalization constant $\sigma$, whereas in the second we collect a remainder term that compares the statistic under the theoretical normalization with the \emph{studentized} version, where the normalization is implemented by means of the sample variance $\widehat{\sigma}$. More precisely, we have the following identity:
\begin{equation*}
A_N(s)=A_{1;N}(s)+A_{2;N}(s)\,,
\end{equation*}
where
\begin{equation*}
A_{1;N}(s)=\frac{1}{\sqrt{N}}\frac{1}{\sigma}%
\sum_{t=1}^{[Ns]}\sum_{\ell =\underline{\ell}}^{\overline{\ell}}%
\sum_{m=-\ell }^{\ell } (\beta _{\ell m}(t)-\widehat{\mu }_{\ell m})
\end{equation*}
and
\begin{equation*}
A_{2;N}(s)=\frac{1}{\sqrt{N}}\left(\frac{1}{\widehat\sigma}-\frac{1}{\sigma}\right)
\sum_{t=1}^{[Ns]}\sum_{\ell =\underline{\ell}}^{\overline{\ell}}\sum_{m=-\ell }^{\ell}
\left( \beta _{\ell m}(t)-\widehat{\mu }_{\ell m}\right).
\end{equation*}

Let us first prove that, under the null hypothesis $H_0$,
\begin{equation}\label{sup-zero}
\sup_{s\in[0,1]}\abs{{A}_{2;N}(s)}=o_{\Prob}(1)\,.
\end{equation}
Indeed,
\begin{flalign*}
\sup_{s\in[0,1]}\abs{{A}_{2;N}(s)}
&=\sup_{s\in[0,1]}\abs{\frac{1}{\sqrt{N}}\left(\frac{1}{\widehat\sigma}-\frac{1}{\sigma}\right)
\sum_{t=1}^{[Ns]}\sum_{\ell =\underline{\ell}}^{\overline{\ell}}\sum_{m=-\ell }^{\ell}
\left( \beta _{\ell m}(t)-\widehat{\mu }_{\ell m}\right)}\\
&=\sup_{s\in[0,1]}\abs{\frac{1}{\sqrt{N}}\frac{1}{\sigma}\left(\frac{\sigma}{\widehat\sigma}-1\right)
\sum_{t=1}^{[Ns]}\sum_{\ell =\underline{\ell}}^{\overline{\ell}}\sum_{m=-\ell }^{\ell}
\left( \beta _{\ell m}(t)-\widehat{\mu }_{\ell m}\right)}\\
&=\abs{\frac{\sigma}{\widehat\sigma}-1}
\sup_{s\in[0,1]}\abs{\frac{1}{\sigma}\frac{1}{\sqrt{N}}\sum_{t=1}^{[Ns]}\sum_{\ell =\underline{\ell}}^{\overline{\ell}}\sum_{m=-\ell }^{\ell}
\left( \beta _{\ell m}(t)-\widehat{\mu }_{\ell m}\right)}\,,
\end{flalign*}
where, under $H_0$, 
$$
\abs{\frac{\sigma}{\widehat\sigma}-1} \to 0  \quad \text{ in probability as $N \to \infty$}
$$
and 
$$
\sup_{s\in[0,1]}\abs{\frac{1}{\sigma}\frac{1}{\sqrt{N}}\sum_{t=1}^{[Ns]}\sum_{\ell =\underline{\ell}}^{\overline{\ell}}\sum_{m=-\ell }^{\ell}
\left( \beta _{\ell m}(t)-\widehat{\mu }_{\ell m}\right)} \implies \sup_{s\in[0,1]}\abs{W(s)}\,,
$$
which gives the desired result.

\subsection{Convergence of the finite-dimensional distributions}

As usual weak convergence will be established by proving convergence of the finite-dimensional distributions and tightness.

In the general case of possibly non-Gaussian sphere-cross-time random fields it is no longer the case that the spherical harmonic processes $a_{\ell m }(t)$ are independent for different values of $\ell$ and $m$; indeed, it was proved in \cite{BaldiMarinucciSPL} that, under isotropy, independence of the spherical harmonic coefficients necessarily implies Gaussianity. However, under isotropy the spherical harmonic processes are still necessarily uncorrelated and their dependence only affects higher order joint moments and cumulants (the so-called polyspectra, see e.g., \cite[Chapter 6]{MaPeCUP}). As a consequence, the computations of the limiting covariance function is basically unaltered in Gaussian and non-Gaussian circumstances; for the finite dimensional distribution it is enough to resort to standard multivariate central limit theorems for stationary vector processes. 

In the sequel we will use the following expression for $A_{1;N}(s)$, which holds under $H_0$,
\begin{eqnarray}\label{eq:alm-bar}
A_{1;N}(s)&=&\frac{1}{\sigma\sqrt{N}}\sum_{t=1}^{[Ns]}\sum_{\ell =\underline{\ell}}^{\overline{\ell}}
\sum_{m=-\ell }^{\ell } (\beta _{\ell m}(t)-\widehat{\mu }_{\ell m})\notag\\
&=&\frac{1}{\sigma\sqrt{N}}\sum_{\ell =\underline{\ell}}^{\overline{\ell}}\sum_{m=-\ell }^{\ell }[Ns]\left (\bar a_{\ell m}(Ns)-\bar a_{\ell m}(N)\right )\,,
\end{eqnarray}
with
\begin{equation*}
\bar a_{\ell m}(Nu):=\begin{dcases}
    \frac1{[Nu]}\sum_{t=1}^{[Nu]}a_{\ell m}(t) &\qquad u\in[1/N,1]\\
    0 &\qquad u\in [0,1/N)
\end{dcases}\,.
\end{equation*}
Again, when writing $s\in[0,1]$ we always implicitly make the previous distinction.

\paragraph{Convergence of the covariance functions.}

We study the covariance function of the leading term $A_{1;N}(s)$, which is easily seen to be given by
\begin{align*}
&\mathbb{E}[A_{1;N}(s)A_{1;N}(s^{\prime })] \\
&= \frac{1}{\sigma^2N}\sum_{\ell,\ell^\prime =\underline{\ell}}^{\overline{\ell}}\sum_{m=-\ell }^{\ell }\sum_{m^{\prime }=-\ell^{\prime }}^{\ell^{\prime }}[Ns][Ns']\E{\left (\bar a_{\ell m}(Ns)-\bar a_{\ell m}(N)\right )\left (\bar a_{\ell' m'}(Ns')-\bar a_{\ell' m'}(N)\right )}\\
&= \frac{1}{\sigma^2 N}\sum_{\ell =\underline{\ell}}^{\overline{\ell}}\sum_{m=-\ell }^{\ell }[Ns][Ns']\E{\left (\bar a_{\ell m}(Ns)-\bar a_{\ell m}(N)\right )\left (\bar a_{\ell m}(Ns')-\bar a_{\ell m}(N)\right )}\\
&\quad+ \frac{1}{\sigma^2 N}\sum_{\ell \ne \ell'}\sum_{m=-\ell }^{\ell }\sum_{m^{\prime }=-\ell^{\prime }}^{\ell^{\prime }}[Ns][Ns']\E{\left (\bar a_{\ell m}(Ns)-\bar a_{\ell m}(N)\right )\left (\bar a_{\ell' m'}(Ns')-\bar a_{\ell' m'}(N)\right )}
\end{align*}
where the second term of the last sum is zero since the $a_{\ell m}$'s are uncorrelated for different $(\ell,m)$, while
\begin{eqnarray*}
&&\E{\left (\bar a_{\ell m}(Ns)-\bar a_{\ell m}(N)\right )\left (\bar a_{\ell m}(Ns')-\bar a_{\ell m}(N)\right )}\\
&&=\E{\bar a_{\ell m}(Ns)a_{\ell m}(Ns')+\bar a_{\ell m}(N)^2-\bar a_{\ell m}(Ns)\bar a_{\ell m}(N)-\bar a_{\ell m}(Ns')\bar a_{\ell m}(N)}\ .
\end{eqnarray*}
Now, for all $u,u' \in [0,1]$ we have that
\begin{align}
&\frac{1}{\sigma^2 N}\sum_{\ell =\underline{\ell}}^{\overline{\ell}}\sum_{m=-\ell }^{\ell }[Ns][Ns']\E{\bar a_{\ell m}(Nu)a_{\ell m}(Nu')}\notag\\
&=\frac{[Ns][Ns']}{\sigma N}\sum_{\ell =\underline{\ell}}^{\overline{\ell}}\frac{(2\ell+1)}{[Nu][Nu']}\sum_{t=1}^{[Nu]}\sum_{t'=1}^{[Nu']} C_{\ell}(t-t')\notag\\
&=\frac{[Ns][Ns']}{\sigma^2 N}\sum_{\ell =\underline{\ell}}^{\overline{\ell}}\frac{(2\ell+1)}{[Nu][Nu']}\left\{\sum_{\tau=-[N(u\wedge u^{\prime })]+1}^{[N(u\wedge u^{\prime})]-1} \!\!\!\!\!\!\!\!\!\!\!\!\left(N(u\wedge u^{\prime })-|\tau|\right)C_\ell(\tau)+ \sum_{t'=1}^{[N(u\wedge u')]} \!\!\!\!\!\!\sum_{t=[N(u\wedge u')]+1}^{[N(u\vee u')]} \!\!\!\!\!\!C_\ell(t-t')\right \}		\notag\\
&=\frac{(u\wedge u')[Ns][Ns']}{\sigma^2}\sum_{\ell =\underline{\ell}}^{\overline{\ell}}\frac{(2\ell+1)}{[Nu][Nu']} \sum_{\tau=-[N(u\wedge u^{\prime })]+1}^{[N(u\wedge u^{\prime })]-1}\left(1-\frac{|\tau|}{N(u\wedge u^{\prime })}
\right)C_\ell(\tau)+o(1) \label{o-piccolo}\\
&= (u\wedge u')\frac{[Ns][Ns']}{[Nu][Nu']}\frac{1}{\sigma^2}\sum_{\ell =\underline{\ell}}^{\overline{\ell}} 2\pi (2\ell+1) f_\ell(0)+o(1) \to (u\wedge u')\frac{s}{u}\frac{s'}{u'}\,. \label{cov}
\end{align}
To prove \eqref{o-piccolo}, we show that for each $\delta>0$ there exists $N_\delta$ such that $R_N<\delta$ for any $N>N_\delta$, where $R_N$ is defined below. Fix $\delta>0$ and consider, for a suitable $M>0$ to be specified,
$$
\delta'=\frac{\delta}{4\sigma^2M}, \qquad  \delta''=\frac{\delta}{4\sigma^2}\,.
$$
Without loss of generality, assume $u'<u$ and that $u'>\delta'>0$; then we have
\begin{align*}
R_N&:=\frac{2}{\sigma^2}\sum_{\ell =\underline{\ell}}^{\overline{\ell}} \frac{2\ell+1}N \sum_{t'=1}^{[Nu']} \sum_{t=[Nu']+1}^{[Nu]} C_\ell(t-t')\\
&=\frac{2}{\sigma^2}\sum_{\ell =\underline{\ell}}^{\overline{\ell}} \frac1N \sum_{t'=1}^{[N(u'-\delta')]} \sum_{t=[Nu']+1}^{[Nu]}  C_\ell(t-t')
+\frac{2}{\sigma^2}\sum_{\ell =\underline{\ell}}^{\overline{\ell}} \frac1N \sum_{t'=[N(u'-\delta')]}^{[Nu']} \sum_{t=[Nu']+1}^{[Nu]}  C_\ell(t-t')\\
&=R_{1;N}+R_{2;N}\,.   
\end{align*}
We note that
\begin{flalign*}
\sum_{t=[Nu']+1}^{[Nu]} \sum_{\ell =\underline{\ell}}^{\overline{\ell}} (2\ell+1)C_\ell(t-t')\le\sum_{\tau=0}^{\infty} \sum_{\ell =\underline{\ell}}^{\overline{\ell}} (2\ell+1) \abs{ C_\ell(\tau)} =: M\,,
\end{flalign*}
and hence
$$
R_{2;N}\le\frac{2}{\sigma^2} \frac1N \sum_{t'=[N(u'-\delta')]}^{[Nu']}M= \frac{2}{\sigma^2} M\,\delta'=\frac{\delta}2\,.
$$
Regarding the first term, we have
$$
R_{1;N}\le\frac{2}{\sigma^2}\sum_{\tau=[N\delta']}^{+\infty} \sum_{\ell =\underline{\ell}}^{\overline{\ell}}  (2\ell+1)\, C_\ell(\tau),
$$
and since $ \displaystyle\sum_{\tau=[N\delta']}^{+\infty} \sum_{\ell =\underline{\ell}}^{\overline{\ell}}  (2\ell+1)\, C_\ell(\tau)$ represents the tail of a convergent series, there exists $N_{\delta''}=N_\delta$ such that
$$
\sum_{\tau=[N\delta']}^{+\infty} \sum_{\ell =\underline{\ell}}^{\overline{\ell}}  (2\ell+1)\, C_\ell(\tau) <  \delta'',
$$
for any $N>N_{\delta}$, and hence
$$
R_{1;N}<\frac{2}{\sigma^2} \delta''=\frac{\delta}2 \ .
$$
We just proved that for any $\delta>0$, there exists $N_{\delta}$ such that
$$
R_N<\delta, \qquad \text{for any } N>N_{\delta}\,,
$$
and hence $R_N=o(1)$. Finally, applying \eqref{cov}, we have
\[
\mathbb{E}[A_{1;N}(s)A_{1;N}(s^{\prime })]
\]
\begin{align*}
&=\frac{[Ns][Ns']}{\sigma^2 N}\sum_{\ell =\underline{\ell}}^{\overline{\ell}}\sum_{m=-\ell }^{\ell } \E{\bar a_{\ell m}(Ns)a_{\ell m}(Ns')+\bar a_{\ell m}(N)^2-\bar a_{\ell m}(Ns)\bar a_{\ell m}(N)-\bar a_{\ell m}(Ns')\bar a_{\ell m}(N)}\\
&\to(s\wedge s')+ss'-ss'-ss'=(s\wedge s')-ss'.
\end{align*}

\paragraph{The central limit theorem.}

To conclude the proof of the central limit theorem we only need to check that all joint cumulants of $A_{1;N}(s)$ of order strictly larger than two converge to zero. Recall that
\begin{equation*}
\bar a_{\ell m}(Nu):=\frac1{[Nu]}\sum_{t=1}^{[Nu]}a_{\ell m}(t)\,, \qquad u\in[0,1]\,.
\end{equation*}
Because the sum in \eqref{eq:alm-bar} is finite, it is enough to exploit a multivariate central limit theorem for the vector
$$
\sqrt{N}\left(\bar a_{\underline{\ell}(-\underline{\ell})}(Ns_{1,1}),\dots, \bar a_{\underline{\ell}\,\underline{\ell}}(Ns_{1,2\underline{\ell}+1}),\dots,\bar a_{\overline{\ell}(-\overline{\ell})}(Ns_{L,1}),\dots,\bar a_{\overline{\ell}\,\overline{\ell}}(Ns_{L,2\overline{\ell}+1})\right)\,,
$$
for any $(s_{1,1},\dots,s_{1,2\underline{\ell}+1},\dots,s_{L,1},\dots,s_{L,2\overline{\ell}+1})$, where $L=\underline{\ell}+\overline{\ell}-1$.

Let us denote $X_1, \dots, X_p$ any $p >2$ elements of this vector with possible repetitions. In view of Assumption \ref{assumption:cum}, we have immediately that
\begin{align*}
&\abs{\operatorname{cum}(X_1, \dots, X_p)}\\
&\le \frac{N^{p/2}}{[Ns_1]\cdots [Ns_p]}\sum_{t_1=1}^{[Ns_1]}\cdots\sum_{t_p=1}^{[Ns_p]}\sum_{\ell_1,\dots,\ell_p}\sum_{m_1,\dots,m_p}\abs{\operatorname{cum}\left(a_{\ell_1 m_1}(t_1),\dots, a_{\ell_p m_p}(t_p)\right)}\\
&=\frac{N^{p/2}}{[Ns_1]\cdots [Ns_{p-1}]}\sum_{\tau_1,\dots,\tau_{p-1}=-\infty}^{\infty}\sum_{\ell_1,\dots,\ell_p}\sum_{m_1,\dots,m_p}\abs{\operatorname{cum}\left(a_{\ell_1 m_1}(\tau_1),\dots,a_{\ell_{p-1} m_{p-1}}(\tau_{p-1}), a_{\ell_p m_p}(0)\right)}\\
&= O(N^{1-p/2}),
\end{align*}
and convergence of the finite-dimensional distributions is hence established. To conclude the proof of weak convergence, we need to establish a tightness condition; this goal is pursued in the following section.

\subsection{Tightness of $A_N(s)$}

We exploit a classical tightness criterion for sequences of processes in $D[0,1]$, as given for instance in \cite{wichura}; we need to prove that there exist $\beta>1$, $\gamma>0$ such that
\begin{equation}\label{cond-tight}
\E{\abs{ A_{1;N}(B_1)}^{\gamma_1}\abs{ A_{1;N}(B_2)}^{\gamma_2}}\le C\, \left (\nu\left (B_1 \cup B_2\right )\right )^\beta\,, \qquad \gamma_1+\gamma_2=\gamma,
\end{equation}
where 
$$
B_1= (s_1,s]\,,      \qquad B_2=(s,s_2]
$$
and $A_{1;N}(B)$ is the increment of the functional $A_{1;N}$ on the interval $B=(s_1,s_2]$ and it is defined as
$$
A_{1;N}(B)=A_{1;N}(s_2)-A_{1;N}(s_1)\,.
$$
We choose $\gamma_1=\gamma_2=2$, to have
\begin{align*}
&\E{\abs{ A_{1;N}(B_1)}^{2}\abs{ A_{1;N}(B_2)}^{2}}\\
&=\E{\abs{\frac{1}{\sigma\sqrt{N}}\sum_{t_1=[Ns_1]+1}^{[Ns]}\sum_{\ell_1 =\underline{\ell}}^{\overline{\ell}}
\sum_{m_1=-\ell_1 }^{\ell_1} a _{\ell_1 m_1}(t_1)}^2\abs{\frac{1}{\sigma\sqrt{N}}\sum_{t_2=[Ns]+1}^{[Ns_2]}\sum_{\ell_2 =\underline{\ell}}^{\overline{\ell}}
\sum_{m_2=-\ell_2}^{\ell_2} a _{\ell_2 m_2}(t_2)}^2}\\
&=\frac{1}{\sigma^4N^2}\sum_{t_1,t_2=[Ns_1]+1}^{[Ns]}\sum_{t_3,t_4=[Ns]+1}^{[Ns_2]}\sum_{\ell_1,\ell_2 =\underline{\ell}}^{\overline{\ell}}
\sum_{m_1=-\ell_1 }^{\ell_1 }\sum_{m_2=-\ell_2 }^{\ell_2 } \E{a _{\ell_1 m_1}(t_1)a _{\ell_1 m_1}(t_2)a _{\ell_2 m_2}(t_3)a _{\ell_2 m_2}(t_4)}.
\end{align*}
Now, we observe that
\begin{align*}
&\E{a _{\ell_1 m_1}(t_1)a _{\ell_1 m_1}(t_2)a _{\ell_2 m_2}(t_3)a _{\ell_2 m_2}(t_4)}\\
&=\E{a _{\ell_1 m_1}(t_1)a _{\ell_1 m_1}(t_2)}\E{a _{\ell_2 m_2}(t_3)a _{\ell_2 m_2}(t_4)}+\E{a _{\ell_1 m_1}(t_1)a _{\ell_2 m_2}(t_3)}\E{a _{\ell_1 m_1}(t_2)a _{\ell_2 m_2}(t_4)}\\
&+\E{a _{\ell_1 m_1}(t_1)a _{\ell_2 m_2}(t_4)}\E{a _{\ell_1 m_1}(t_2)a _{\ell_2 m_2}(t_3)}+\operatorname{cum}\big(a _{\ell_1 m_1}(t_1),a _{\ell_1 m_1}(t_2),a _{\ell_2 m_2}(t_3),a _{\ell_2 m_2}(t_4)\big)\\
&=C_{\ell_1} (t_2-t_1)C_{\ell_2}(t_4-t_3)+C_{\ell_1} (t_3-t_1)C_{\ell_1}(t_4-t_2)\1_{\ell_1=\ell_2,m_1=m_2}\\
&+C_{\ell_1} (t_4-t_1)C_{\ell_1}(t_3-t_2)\1_{\ell_1=\ell_2,m_1=m_2}+\operatorname{cum}\big(a _{\ell_1 m_1}(t_1),a _{\ell_1 m_1}(t_2),a _{\ell_2 m_2}(t_3),a _{\ell_2 m_2}(t_4)\big)
\end{align*}
and by stationarity we have that
\begin{align*}
&\operatorname{cum}\big(a _{\ell_1 m_1}(t_1),a _{\ell_1 m_1}(t_2),a _{\ell_2 m_2}(t_3),a _{\ell_2 m_2}(t_4)\big)\\
&= \operatorname{cum}\big(a _{\ell_1 m_1}(0),a _{\ell_1 m_1}(t_2-t_1),a _{\ell_2 m_2}(t_3-t_1),a _{\ell_2 m_2}(t_4-t_1)\big).
\end{align*}
Hence, 
\begin{align*}
&\E{\abs{ A_{1;N}(B_1)}^{2}\abs{ A_{1;N}(B_2)}^{2}}\\
&\le \frac{1}{\sigma^4N^2} \left(\sum_{\ell=\underline{\ell}}^{\overline{\ell}}(2\ell+1) \sum_{t_1,t_2=[Ns_1]+1}^{[Ns_2]}\, \abs{C _{\ell}(t_2-t_1)} \right)^2\\
 & +\frac{2}{\sigma^4N^2} \sum_{\ell =\underline{\ell}}^{\overline{\ell}}
 (2\ell+1) \left ( \sum_{t_1,t_2=[Ns_1]+1}^{[Ns_2]} \abs{C _{\ell}(t_2-t_1)} \right)^2\\
 & +\frac{1}{\sigma^4N^2}\sum_{t_1,t_2,t_3,t_4=[Ns_1]+1}^{[Ns_2]}\sum_{\ell_1,\ell_2 =\underline{\ell}}^{\overline{\ell}}
\sum_{m_1=-\ell_1 }^{\ell_1 }\sum_{m_2=-\ell_2 }^{\ell_2 }\, \abs{\operatorname{cum}\big(a _{\ell_1 m_1}(0),a _{\ell_1 m_1}(t_2-t_1),a _{\ell_2 m_2}(t_3-t_1),a _{\ell_2 m_2}(t_4-t_1)\big)}\,.
\end{align*}
In particular, in view of Assumption \ref{assumption:cum}, for the last term we get
\begin{align*}
&\sum_{t_1,t_2,t_3,t_4=[Ns_1]+1}^{[Ns_2]}\sum_{\ell_1,\ell_2 =\underline{\ell}}^{\overline{\ell}}
\sum_{m_1=-\ell_1 }^{\ell_1 }\sum_{m_2=-\ell_2 }^{\ell_2 }\,\abs{\operatorname{cum}\big(a _{\ell_1 m_1}(0),a _{\ell_1 m_1}(t_2-t_1),a _{\ell_2 m_2}(t_3-t_1),a _{\ell_2 m_2}(t_4-t_1)\big)}\\
&\le ([Ns_2]-[Ns_1])\sum_{\tau_1,\tau_2,\tau_3=-\infty}^{\infty}\sum_{\ell_1,\ell_2 =\underline{\ell}}^{\overline{\ell}}
\sum_{m_1=-\ell_1 }^{\ell_1 }\sum_{m_2=-\ell_2 }^{\ell_2 }\,\abs{\operatorname{cum}\big(a _{\ell_1 m_1}(0),a _{\ell_1 m_1}(\tau_1),a _{\ell_2 m_2}(\tau_2),a _{\ell_2 m_2}(\tau_3)\big)}\\
&\le ([Ns_2]-[Ns_1])^2\sum_{\ell_1,\ell_2 =\underline{\ell}}^{\overline{\ell}}
\sum_{m_1=-\ell_1 }^{\ell_1 }\sum_{m_2=-\ell_2 }^{\ell_2 }\sum_{\tau_1,\tau_2,\tau_3=-\infty}^{\infty}\,\abs{\operatorname{cum}\big(a _{\ell_1 m_1}(0),a _{\ell_1 m_1}(\tau_1),a _{\ell_2 m_2}(\tau_2),a _{\ell_2 m_2}(\tau_3)\big)}\,,
\end{align*}
where the last inequality follows from the fact that $[Ns_2]-[Ns_1]$ is an entire number and hence $[Ns_2]-[Ns_1]\le ([Ns_2]-[Ns_1])^2$.
Then, since the sum over multipoles is finite, we obtain
\begin{align*}
&\sum_{t_1,t_2,t_3,t_4=[Ns_1]+1}^{[Ns_2]}\sum_{\ell_1,\ell_2 =\underline{\ell}}^{\overline{\ell}}
\sum_{m_1=-\ell_1 }^{\ell_1 }\sum_{m_2=-\ell_2 }^{\ell_2 }\,\abs{\operatorname{cum}\big(a _{\ell_1 m_1}(0),a _{\ell_1 m_1}(t_2-t_1),a _{\ell_2 m_2}(t_3-t_1),a _{\ell_2 m_2}(t_4-t_1)\big)}\\
&\le M \,([Ns_2]-[Ns_1])^2 \le 4 \, M \, N^2 \,(s_2-s_1)^2 \,,
\end{align*}
for some $M>0$, using the same argument as in \cite[Theorem 14.1]{Billingsley}. The other terms in the sum can be dealt with similarly.

%$$
%\abs{\operatorname{cum}_p(a_{\ell_1 m_1}(t_1),a_{\ell_2 m_2}(t_2),a_{\ell_3 m_3}(t_3),a_{\ell_4 m_4}(t_4))} < const \times \prod_{i \neq j} (t_i-t_j)^{-\gamma}
%$$

\section{Proof of Theorem \ref{thm-H1}}\label{proof:thm-H1}

We first note that, under the alternative hypothesis $H_1$,
\begin{flalign*}
\beta_{\ell m}(t)-\widehat\mu_{\ell m}=a_{\ell m }(t)-\frac 1N\sum_{u=1}^N a_{\ell m}(u)+\left (\mu _{\ell m}(t) -\frac 1N\sum_{u=1}^N\mu _{\ell m}(u)\right ),
\end{flalign*}
and hence 
\begin{flalign*}
A_N(s)&=\frac{1}{\widehat{\sigma}}\underbrace{\frac1{\sqrt{N}}\sum_{t=1}^{[Ns]} \sum_{\ell=\underline{\ell}}^{\overline{\ell}} \sum_{m=-\ell}^\ell \left(a_{\ell m }(t)-\frac 1N\sum_{u=1}^N a_{\ell m}(u)\right)}_{B_N(s)}\\
&\quad+\frac{1}{\widehat{\sigma}}\underbrace{\frac1{\sqrt{N}}\sum_{t=1}^{[Ns]} \sum_{\ell=\underline{\ell}}^{\overline{\ell}} \sum_{m=-\ell}^\ell \left (\mu _{\ell m}(t) -\frac 1N\sum_{u=1}^N\mu _{\ell m}(u)\right )}_{E_N(s)}\\
&=\frac{B_N(s)}{\widehat\sigma}+\frac{E_N(s)}{\widehat\sigma}\ .
%+\left(\frac{\sigma}{\widehat{\sigma}}-1\right)\left(\frac{B_N(s)}{\sigma}+\frac{E_N(s)}{\sigma}\right)\,.
\end{flalign*}
Now, as $N \to \infty$, 
\begin{flalign*}
B_N(s) \implies \sigma W(s)\,,
\end{flalign*}
as a result of Theorem \ref{thm-H0}. Moreover, 
\begin{align*}
E_N(s)&=\sqrt{N}\frac1{N}\sum_{t=1}^{[Ns]} \sum_{\ell=\underline{\ell}}^{\overline{\ell}} \sum_{m=-\ell}^\ell \left (\mu _{\ell m}(t) -\frac 1N\sum_{u=1}^N\mu _{\ell m}(u)\right )\\
&=\sqrt{N}\frac1{N}\sum_{t=1}^{[Ns]} \sum_{\ell=\underline{\ell}}^{\overline{\ell}} \sum_{m=-\ell}^\ell \left (N^{\alpha_\ell} g_{\ell m}\left(\frac t N\right) -\frac 1N\sum_{u=1}^NN^{\alpha_\ell} g_{\ell m}\left(\frac u N\right)\right )\\
&= N^{\overline{\alpha}+1/2}
\left(\sum_{\ell\in \overline{\mathcal{I}}}\sum_{m=-\ell}^\ell  \left (\int_0^s g_{\ell m}(t) dt - s\int_0^1 g _{\ell m}(u)du\right ) + o(1) \right),
\end{align*}
where $o(1)$ is uniform in $s$ (since we assume $g_{\ell m}$ to be piecewise continuous and bounded on [0,1], see e.g.~\cite[Lemma 4]{KPA:88}). %\textcolor{red}{(CHECK)} 
Note that, 
\begin{itemize}
    \item if $\overline{\alpha} <-1/2$, $N^{\overline{\alpha}+1/2}\to 0$, as $N\to \infty$; 
    \item if $\overline{\alpha} = -1/2$, $N^{\overline{\alpha}+1/2}=1 $;
    \item  if $\overline{\alpha} > -1/2$, $N^{\overline{\alpha}+1/2}\to \infty$, as $N\to \infty$.
\end{itemize}
As a consequence, the final result on the test statistic is as follows: for $\overline{\alpha} < -1/2$, the statistic converges weakly to the Brownian bridge $W$ (as  under $H_0$),
while if $\overline{\alpha} = -1/2$,  $A_N(s)$ converges weakly to a shifted Brownian bridge, i.e.,
$$
A_N(s) \Longrightarrow W(s) + \frac{1}{\sigma}
\sum_{\ell\in \overline{\mathcal{I}}}\sum_{m=-\ell}^\ell  \left (\int_0^s g_{\ell m}(t) dt - s\int_0^1 g _{\ell m}(u)du\right ).
$$

Case \emph{(II)} requires a more careful inspection. First assume $q_N N^{2\overline{\alpha}} \to c\ge0$; then, for fixed $s\in[0,1]$, we have
\begin{align*}
\frac{E_N(s)}{\widehat\sigma}&=N^{\overline{\alpha}+1/2} \frac{
\sum_{\ell\in \overline{\mathcal{I}}}\sum_{m=-\ell}^\ell  \left (\int_0^s g_{\ell m}(t) dt - s\int_0^1 g _{\ell m}(u)du\right ) + o(1) }{\left(\sigma^2+q_N N^{2\overline{\alpha}} \sum_{\ell\in \overline{\mathcal{I}}}\sum_{m=-\ell}^\ell \Var(g_{\ell m}(U)+O_\Prob(q_N N^{\overline{\alpha}-1/2})\right)^{1/2}}\\
&=N^{\overline{\alpha}+1/2}\left(\frac{\sum_{\ell\in \overline{\mathcal{I}}}\sum_{m=-\ell}^\ell  \left (\int_0^s g_{\ell m}(t) dt - s\int_0^1 g _{\ell m}(u)du\right )}{\left(\sigma^2+c \sum_{\ell\in \overline{\mathcal{I}}}\sum_{m=-\ell}^\ell \Var(g_{\ell m}(U)\right)^{1/2}}+o_\Prob(1)\right).
\end{align*}
Now assume that $q_N N^{2\overline{\alpha}} \to +\infty$; then, for fixed $s\in[0,1]$, we have
\begin{align*}
\frac{E_N(s)}{\widehat\sigma}&=\frac{E_N(s)}{\sqrt{q_N} N^{\overline{\alpha}}}\frac{\sqrt{q_N} N^{\overline{\alpha}}}{\widehat\sigma} \\
&=\sqrt{\frac{N}{q_N}}\left(\frac{\sum_{\ell\in \overline{\mathcal{I}}}\sum_{m=-\ell}^\ell  \left (\int_0^s g_{\ell m}(t) dt - s\int_0^1 g _{\ell m}(u)du\right )}{\left(\sum_{\ell\in \overline{\mathcal{I}}}\sum_{m=-\ell}^\ell \Var(g_{\ell m}(U)\right)^{1/2}}+o_\Prob(1)\right)
\end{align*}
Note that when $\overline{\alpha} \ge 0$, necessarily $q_N N^{2\overline{\alpha}} \to +\infty$, while when $\overline{\alpha} \in (-1/2,0)$ the behavior of $q_N N^{2\overline{\alpha}}$ depends on the choice of $q_N$.

As a consequence, we just proved that for any $s\in[0,1]$
$$
A_N(s)=c(s) r_N(\overline{\alpha})+ o_\Prob\left( r_N(\overline{\alpha})\right)\,,
$$
with
\begin{equation}\label{eq:cs}
c(s) =\begin{dcases}\frac{\sum_{\ell\in \overline{\mathcal{I}}}\sum_{m=-\ell}^\ell  \left (\int_0^s g_{\ell m}(t) dt - s\int_0^1 g _{\ell m}(u)du\right )}{\left(\sigma^2+c \sum_{\ell\in \overline{\mathcal{I}}}\sum_{m=-\ell}^\ell \Var(g_{\ell m}(U)\right)^{1/2}} &\text{when } q_N N^{2\overline{\alpha}} \to c \ge 0  \\
 \frac{\sum_{\ell\in \overline{\mathcal{I}}}\sum_{m=-\ell}^\ell  \left (\int_0^s g_{\ell m}(t) dt - s\int_0^1 g _{\ell m}(u)du\right )}{\left(\sum_{\ell\in \overline{\mathcal{I}}}\sum_{m=-\ell}^\ell \Var(g_{\ell m}(U)\right)^{1/2}} &\text{when } q_N N^{2\overline{\alpha}} \to +\infty
\end{dcases},
\end{equation}
and hence
\begin{equation}\label{eq:cs-prob}
\frac{\abs{A_N(s)}}{r_N(\overline{\alpha})}-\abs{c(s)} \to 0 \quad \text{in probability} \quad \text{as} \quad N\to \infty.
\end{equation}
Now we choose $s^\star \in [0,1]$ such that $c(s^\star)\ne 0$ and we exploit the fact that
\begin{flalign*}
&\Prob\left(\sup_s\abs{A_N(s)}> K\cdot r_N(\overline{\alpha})\right)\ge \Prob\bigg(\abs{A_N(s^\star)}>  K \cdot r_N(\overline{\alpha})\bigg)\,.
\end{flalign*}
Since \eqref{eq:cs-prob}, for any $\eps>0$,
$$
\Prob \left( \abs{\frac{\abs{A_N(s^\star)}}{r_N(\overline{\alpha})}-\abs{c(s^\star)}}\le \eps\right)\to 1\,,
$$
as $N\to\infty$; moreover,
\begin{flalign*}
&\Prob \left( \abs{\frac{\abs{A_N(s^\star)}}{r_N(\overline{\alpha})}-\abs{c(s^\star)}}\le \eps\right)=\Prob \left(\abs{c(s^\star)}-\eps \le \frac{\abs{A_N(s^\star)}}{r_N(\overline{\alpha})}\le \abs{c(s^\star)}+\eps\right)\\
&\le \Prob \left(\frac{\abs{A_N(s^\star)}}{r_N(\overline{\alpha})}\ge \abs{c(s^\star)}-\eps\right)= \Prob \bigg(\abs{A_N(s^\star)}\ge \left(\abs{c(s^\star)}-\eps\right)r_N(\overline{\alpha})\bigg)
\end{flalign*}
and the proof is completed choosing $K=\abs{c(s^\star)}/2$.

\section*{Acknowledgments} DM is grateful to the MUR Department of Excellence Programme \emph{MatModToV} and to Mur \emph{Grafia} for financial support. AV is supported by the co-financing of the European Union - FSE-REACT-EU, PON Research and Innovation 2014-2020, DM 1062/2021.

\bibliographystyle{plain}
\bibliography{biblio}

\end{document}